\pdfoutput=1

	\documentclass[11pt]{article}
	
	\usepackage{url}
    \usepackage{verbatim}
	\usepackage{graphicx}
	\usepackage{amsmath}
	\usepackage{amsthm}
	\usepackage{amsfonts}
    \usepackage{nicefrac}
    \usepackage{longtable}
    \usepackage{subfig}
    \usepackage{rotating}
    \usepackage{color}
	\catcode`\@=11
	\@addtoreset{equation}{section} 
	\renewcommand{\theequation}{\arabic{section}.\arabic{equation}}

	\newtheorem{thm}{Theorem}[section]
	\newtheorem{prop}[thm]{Proposition}

    \newcommand{\Dh}{\Delta_h}

    \newcommand{\hf}{\frac{1}{2}}
    \newcommand{\nrm}[1]{\left\| #1 \right\|}
    
    \newcommand{\cip}[2]{\left( #1 \middle| #2 \right)}
    \newcommand{\eipns}[2]{\left[ #1 \middle\| #2 \right]_{\rm ns}}
    \newcommand{\eipew}[2]{\left[ #1 \middle\| #2 \right]_{\rm ew}}
    
    \newcommand{\ciptwo}[2]{\left( #1 \middle\| #2 \right)}
    \newcommand{\be}{\begin{eqnarray}}
    \newcommand{\ee}{\end{eqnarray}}
	
\begin{document}

\title{Energy Stable and Efficient Finite-Difference Nonlinear Multigrid Schemes for the Modified Phase Field Crystal Equation}

	\author{
{\bf Arvind Baskaran}
	\\
Mathematics Department 
	\\ 
The University of California; Irvine, CA 92697
	\\ 
\url{baskaran@math.uci.edu}
        \\
{\bf Peng Zhou}\footnote{Currently at Haerbin Inst. Techn., China}
	\\
Mathematics Department 
	\\ 
University of California; Irvine, CA 92697
	\\
\url{zhoup@math.uci.edu}
	\\
{\bf Zhengzheng Hu}
	\\
Mathematics Department
	\\
North Carolina State University; Raleigh, NC 27695
	\\
\url{zhu4@ncsu.edu}
        \\
{\bf Cheng Wang}
	\\
Mathematics Department
	\\
University of Massachusetts, North Dartmouth, MA 02747
	\\
\url{cwang1@umassd.edu}
	\\
{\bf Steven M. Wise}
	\\
Mathematics Department
	\\
The University of Tennessee, Knoxville, TN 37996
	\\
\url{swise@math.utk.edu}
	\\
{\bf John S. Lowengrub}
	\\
Mathematics Department 
	\\ 
University of California; Irvine, CA 92697 
	\\ 
Corresponding author: \url{lowengrb@math.uci.edu}}

	\maketitle

	\begin{abstract}
In this paper we present two unconditionally energy stable finite difference schemes for the Modified Phase Field Crystal (MPFC) equation, a sixth-order nonlinear damped wave equation, of which the purely parabolic Phase Field Crystal (PFC) model can be viewed as a special case.  The first is a convex splitting scheme based on an appropriate decomposition of the discrete energy and is first order accurate in time and second order accurate in space.  The second is a new, fully second-order scheme that also respects the convex splitting of the energy. Both schemes are nonlinear but may be formulated from the gradients of strictly convex, coercive functionals.  Thus, both are uniquely solvable regardless of the time and space step sizes.  The schemes are solved by efficient nonlinear multigrid methods.  Numerical results are presented demonstrating the accuracy, energy stability, efficiency, and practical utility of the schemes.  In particular, we show that our multigrid solvers enjoy optimal, or nearly optimal complexity in the solution of the nonlinear schemes. 
	\end{abstract}


	\section{Introduction}
	\label{sec-introduction}
Crystalline materials are used in numerous engineering applications.  In practice, most crystals have nanoscopic imperfections in the form of defects -- such as vacancies, grain boundaries, and dislocations -- and controlling, or at least predicting, the formation and evolution of such imperfections is a major challenge.   Accurate modeling of crystal dynamics, especially defect dynamics, requires atomic-scale resolution.  The Phase Field Crystal (PFC) equation, a continuum model that has attracted significant attention in recent years, has shown promise in this regard~\cite{elder04, elder02, provotas07}. In the PFC approach the crystal is described via a continuous phase field $\phi$ that approximates the number density of atoms. The field variable $\phi$ admits two qualitatively different solutions, a constant solution that represents the liquid phase and a periodic solution that represents the solid phase. The local extrema of the periodic solution lie on a near-perfect lattice, mimicking the atomic lattice of the crystal.  The PFC equation is  	
	\begin{equation}
	\label{pfc_dynamics}
\partial_t \phi =  M\Delta \left(\phi^3 + (1-\epsilon) \phi + 2\Delta \phi +\Delta^2 \phi \right) ,
	\end{equation}
where $\epsilon < 1$ and $M>0$.  This model, which is related to the dynamic density functional theory of freezing~\cite{elder04,marconi99}, has the important advantage over many atomistic models that the characteristic time is determined by the diffusion time scale and not by that of the atomic vibrations.  This enables one to capture relatively long time processes, in contrast to atomistic models.  The recent review by Provatas~\emph{et al.}~\cite{provotas07} outlines a wide range of applications of the PFC modeling framework.

The major disadvantage of the PFC model is that it fails to distinguish between the elastic relaxation and diffusion time scales~\cite{elder04,stefanovic06}.  In~\cite{stefanovic06,stefanovic09}, the authors introduced the Modified Phase Field Crystal (MPFC) model to overcome this difficulty.  The MPFC equation is
	\begin{equation}
	\label{mpfc_dynamics}
\partial_{tt} \phi + \beta \partial_t \phi = M\Delta \left(\phi^3 + (1-\epsilon) \phi + 2\Delta \phi +\Delta^2 \phi \right),
	\end{equation}
where $\beta >0$.  The MPFC equation (\ref{mpfc_dynamics}) is a nonlinear damped wave equation modeling a viscoelastic response to perturbations
to the density field.  In this model, perturbations in the density field are transmitted by waves that travel essentially undamped up to a certain length scale determined by the parameters.  When this length scale is of the order of the size of the system a separation of elastic relaxation and diffusion time scales may be practically observed~\cite{stefanovic06}.

The MPFC equation, like the PFC equation, is a sixth order evolutive  nonlinear partial differential equation that cannot be solved analytically in practical circumstances.  Efficient and accurate methods for approximating the solutions of the MPFC equation are therefore highly desirable.  However, to date, very little research has been carried out in this direction.

Because of the close relationship between the MPFC and PFC models, methods for the latter equation can be adapted and applied to the former.  See, for example, \cite{backofen07,cheng08,elder04,hu09,mellenthin08,wise09}  for some recent approximation methods specifically for the PFC model.  However, one must take care, as we show in the following pages, to adequately account for the wave-like nature of the MPFC solutions in the numerical method, especially in the design of provably energy stable schemes.

Methods specifically designed for the MPFC equation can be found in~\cite{galenko11,stefanovic09,wang10b,wang11}.   Stefanovic~\emph{et al.},~\cite{stefanovic09} employed a semi-implicit finite difference discretization, with a multigrid algorithm for solving the algebraic equations.  They provide no numerical analysis for their scheme, which is significantly different from schemes we propose and analyze.  In particular, theirs is not expected to be unconditionally solvable or energy stable.  The authors did, however, give some evidence of the efficiency of their multigrid solver.  We will conduct a similar study for our multigrid solver in this paper.  The MPFC scheme in~\cite{galenko11} is more or less the same as the first-order convex-splitting that we devised earlier in~\cite{wang10b, wang11}.  The first-order convex splitting scheme in our work~\cite{wang10b, wang11} has two fundamental properties.  It is unconditionally energy stable and unconditionally uniquely solvable.  We rigorously analyzed this convex splitting scheme in~\cite{wang10b,wang11}, but we did not provide a practical solution strategy.  That gap is filled here.

In~\cite{hu09,wise09} we presented first and second-order accurate (in time) finite difference schemes for the purely parabolic PFC model, based on a convex splitting framework applied to the physical energy.  The convexity splitting idea -- in the context of first order (in time) convex splitting schemes -- is generally credited to~\cite{eyre98}.  One of the main advances in~\cite{hu09} -- and the more recent work in~\cite{shen11} -- was the demonstration that the convex splitting framework can be extended to second-order (in time) methods as well, in a natural way.  While the motivation in~\cite{hu09} was the application to the PFC model, the second-order convex splitting idea that we developed is, in fact, rather general. 

The main goal of this paper is to apply the first and second-order convex splitting framework to the MPFC equation.  The principal challenge in doing so is that the extension to damped wave dynamics appears, at first sight, not so straightforward.  This obstacle was first surmounted in \cite{wang10b,wang11}, as already mentioned, where we extended the first-order convexity splitting idea for the MPFC model.  Specifically, we showed the existence of a pseudo energy (different from the physical energy), which is  non-increasing in time for solutions of the MPFC model and to which the usual convexity splitting idea can be applied.  As a side note, in~\cite{wang10b} we used this energy stability approach to prove the existence and uniqueness of a global in time smooth solution for the MPFC partial differential equation.

Herein we present, for the first time, a new fully second-order convex splitting scheme for the MPFC model.  The idea is general and can be applied to other damped wave equations.  We also compare the first-order convex splitting scheme from~\cite{wang10b,wang11} with the new second-order convex splitting scheme.  We propose and test efficient (optimal complexity) nonlinear multigrid solvers for both.  These two convex splitting schemes are shown to be unconditionally energy stable with respect to a pseudo energy, which is different from the physical energy, and mass conserving.  Both schemes are nonlinear.  But both are obtained via gradients of strictly convex, coercive functionals, facts which guarantee the unique solvability of the schemes, regardless of the time and space step sizes.

The paper is organized as follows. The MPFC model is recounted in Sec.~\ref{sec-model}.  In Sec.~\ref{sec-schemes} we present our first and second-order accurate (in time) convex splitting schemes and analyze their basic properties.  We present some numerical results in Sec.~\ref{sec-numerical-results} that give evidence of the convergence of the schemes and the efficiency of the multigrid solvers.  In that section, we also show a couple of practical application of the model.  The first is a problem of crystalization, the second, a problem of elastic relaxation in a strained crystal.  We give some conclusions in Sec.~\ref{sec-conclusions}.  To keep the presentation short, we relegate some of the details of the schemes to two Appendices.  In Appendix~\ref{app-summation-by-parts}, we give the basics of our finite difference discretization of space, and we list some needed summation-by-parts formulae.  In Appendix~\ref{app-nonlinear-multigrid}, we give some details of the nonlinear multigrid solvers that we utilize to advance the nonlinear schemes in time.

	\section{The Modified Phase Field Crystal Model}
	\label{sec-model}

Consider the energy 
	\begin{equation}
E ( \phi ) = \int_{\Omega} \left \{ \frac{1}{4} \phi^4 + \frac{\alpha}{2} \phi^2 - |\nabla \phi |^2 + \frac{1}{2} (\Delta \phi)^2 \right \} d\mathbf{x} ,
	\label{energy}
	\end{equation}
where $\phi: \Omega \subset \mathbb{R}^2 \rightarrow \mathbb{R}$ is the atomic density field; and $\alpha:=1-\epsilon$ with $\epsilon \leq 1$.  We assume $\Omega = (0,L_x) \times (0,L_y)$ and $\phi$ is $\Omega$--periodic.  The MPFC equation is the pseudo-gradient flow
	\begin{equation}
\partial_{tt}\phi + \beta \partial_t \phi = \nabla \cdot \left( M(\phi ) \nabla \mu \right) ,
	\end{equation}
where $ M(\phi ) >0 $ is a mobility, and
	\begin{equation}
\mu := \delta_{\phi} E( \phi ) = \phi^3 +\alpha\phi +2\Delta\phi+\Delta^2\phi .
	\end{equation}
The variable $\mu$ is called the chemical potential.  For simplicity, we take $M$ to be a constant and consider the simpler equation
	\begin{equation}
\partial_{tt}\phi + \beta \partial_t \phi =  M\Delta \mu , \quad \mu = \phi^3 +\alpha\phi +2\Delta\phi+\Delta^2\phi .
	\label{mpfc_1}
\end{equation}
Note that in~\cite{wang10b,wang11}, we used an alternate version of the equation, namely, 
	\begin{displaymath}
\tilde{\beta}\partial_{tt}\phi +  \partial_t \phi =  \Delta \mu , \quad \tilde{\beta} >0 ,
	\end{displaymath}
which expresses the model as a perturbed parabolic equation.  For the present exposition we will abandon this form in favor of Eq.~(\ref{mpfc_1}), since (\ref{mpfc_1}) is more widely used in the physics literature~\cite{stefanovic06,stefanovic09}.

It turns out that along the solution trajectories of Eq.~(\ref{mpfc_1}) the energy (\ref{energy}) may actually increase in time, on some time intervals.  (See Fig.~\ref{fig2} and the corresponding discussion in the Sec.~\ref{subsec-energy-stability}.) However, solutions of the MPFC equation do dissipate a pseudo energy~\cite{wang10b,wang11}.  Also observe that Eq.~(\ref{mpfc_1}) is not precisely a mass conservation equation due to the term $\partial_{tt}\phi$.  However, with the introduction of a reasonable initial condition for $\partial_t\phi$, it is possible to show that $\displaystyle{\int_\Omega \partial_t\phi\, d{\bf x} = 0}$ for all time.

We shall need a precise definition of the $H_{per}^{-1}$ inner-product to define an appropriate pseudo energy for the MPFC equation.  Suppose $f\in \left\{u\in L^2\left(\Omega \right) \ \middle| \ \int_\Omega u\, d{\bf x} = 0\right\} =: L_0^2\left(\Omega \right)$. Define $\psi_f\in H_{per}^2\left(\Omega\right)\cap L_0^2\left(\Omega \right)$ to be the unique solution to the PDE problem
	\begin{eqnarray}
-\Delta\psi_f = f\quad \mbox{in}\  \Omega  .
	\end{eqnarray}
In this case we write $\psi_f = -\Delta^{-1} f$.  Suppose that $f,g\in L_0^2\left(\Omega\right)$, then we define
	\begin{equation}
\left(f,g\right)_{H^{-1}} := \left(\nabla\psi_f,\nabla\psi_g \right)_{L^2} .
	\end{equation}
Note that, via integration by parts, we have the equalities
	\begin{equation}
\left(f,g\right)_{H^{-1}}	=  -\left(\Delta^{-1}f,g \right)_{L^2} = -\left(\Delta^{-1}g,f \right)_{L^2} = \left(g,f\right)_{H^{-1}}  .
	\end{equation}
For every $f\in L_0^2\left(\Omega\right)$, we define 	
	\begin{equation}
\nrm{f}_{H^{-1}}=\sqrt{\left(f,f\right)_{H^{-1}}}  .
	\end{equation}
	
We now recast the MPFC Eq.~(\ref{mpfc_1}) as the following system of equations:
	\begin{equation}
\partial_t\psi = M\Delta\mu - \beta\psi \ , \quad \partial_t\phi =: \psi  .
	\label{first-order-system}
	\end{equation}
And we introduce the pseudo energy
	\begin{equation}
{\mathcal E}(\phi,\psi) :=E(\phi) + \frac{1}{2M} \nrm{\psi}^2_{H^{-1}}  ,
	\label{energy-plus-time}
	\end{equation} 
which requires that $\int_\Omega\psi\, d{\bf x} =0$.  This is the case as long as the initial condition $\int_\Omega\partial_t\phi({\bf x},0)\, d{\bf x} = 0$ is satisfied~\cite{wang10b,wang11}.  In what follows, we will use the initial data
	\begin{equation}
\psi({\bf x},0) = \partial_t\phi({\bf x},0) \equiv 0\quad\mbox{in}\quad \Omega  ,
	\end{equation}
so that $\int_\Omega\partial_t\phi({\bf x},0)\, d{\bf x} = 0$ is trivially satisfied.  In any case, as long as $\psi=\partial_t\phi$ is of mean zero, a formal calculation shows that (see~\cite{wang10b,wang11})
	\begin{equation}
d_t{\mathcal E}	= \left(\mu,\partial_t\phi\right)_{L^2} +\frac{1}{M}\left(\psi,\partial_t\psi\right)_{H^{-1}} = -\frac{\beta}{M}\left( \psi, \psi \right)_{H^{-1}} \le 0  ,
	\label{pde-psuedo-energy-decrease}
	\end{equation}
which guarantees that the pseudo energy is non-increasing in time.

At this point we observe that the energy $E(\phi) $ admits a splitting into purely convex and concave energies, given by $E = E_c - E_e$, where $E_c$ and $E_e$ are both convex functionals.  In particular, we have the convex splitting of the form 
	\begin{equation}
E_c = \frac{1}{4} \nrm{\phi}^4_{L^4} + \frac{\alpha}{2} \nrm{\phi}^2_{L^2} + \frac{1}{2}\nrm{\Delta \phi}^2_{L^2}, \quad E_e = \nrm{\nabla\phi}^2_{L^2}.
	\end{equation}
This observation is the key toward the construction of a (first-order) convex splitting scheme, as was demonstrated in \cite{wang10b,wang11}.  It will also be the basis of the new fully second-order scheme that we present in the following sections.

	\section{Finite Difference Schemes and Their Properties}
	\label{sec-schemes}

In this section we present two finite difference methods for the MPFC equations. The first is a first-order accurate convex splitting scheme that was introduced and analyzed in \cite{wang10b,wang11}.  The second is a new, fully second-order accurate convex splitting scheme for the MPFC equation that is similar to a scheme we presented earlier for the PFC equation~\cite{hu09}.  In the following we will show that both schemes are discretely mass conserving and unconditionally energy stable. In addition, both schemes are unconditionally uniquely solvable.  We refer the reader to App.~\ref{app-summation-by-parts} and also~\cite{wise10,wise09} for a description of the finite difference framework that  we employ to define and analyze our schemes.

We begin by defining discrete versions of the energy (\ref{energy}) and pseudo energy (\ref{energy-plus-time}).  For the finite difference grid function $\phi \in \mathcal{C}_{\overline{m}\times\overline{n}}$ define
	\begin{equation}
	\label{discrete_energy_PFC}
F(\phi) := \frac{1}{4} \nrm{\phi}_4^4 + \frac{\alpha}{2} \nrm{\phi}^2_2 - \nrm{\nabla_h \phi}^2_2 
+\frac{1}{2} \nrm{\Delta_h \phi}^2_2.
	\end{equation}
For $\phi \in \mathcal{C}_{\overline{m}\times\overline{n}}$ and $\psi\in H$, define the discrete pseudo energy as
	\begin{equation}
\mathcal{F}(\phi,\psi) := F(\phi) + \frac{1}{2M} \nrm{\psi}_{-1}^2,
	\label{discrete_energy_MPFC}
	\end{equation}
where $H$ is the space of mean zero cell centered functions and the ``$-1$" norm is a finite-dimensional equivalent to the continuous $H^{-1}$ norm.  See App.~\ref{app-summation-by-parts}, in particular (\ref{mean-zero-space}) -- (\ref{H-norm}), and also~\cite{wang11}.  Note that the discrete energy $F$  admits a convex splitting.  In particular, if the grid function $\phi \in \mathcal{C}_{\overline{m}\times\overline{n}}$  is periodic, then the energies
	\begin{eqnarray}
F_c(\phi) := \frac{1}{4} \nrm{\phi}^4_4 + \frac{\alpha}{2} \nrm{\phi}^2_2+ \frac{1}{2}\nrm{\Delta_h \phi}^2_2, \quad F_e(\phi) := \nrm{\nabla_h \phi }^2_2
	\end{eqnarray}
are convex, and $F = F_c-F_e$.  Furthermore, the discrete variations (gradients) are $\delta_{\phi} F_c(\phi) = \phi^3 + \alpha \phi + \Delta_h^2 \phi$ and $\delta_{\phi} F_e(\phi) = -2 \Delta_h \phi$.  

	\subsection{First-Order Convex Splitting Scheme}
	\label{subsec-1st-order-scheme}

The first-order convex-splitting scheme, introduced in~\cite{wang10b,wang11}, is constructed by respecting the convexity splitting $F = F_c-F_e$ and is defined as follows: find $\phi^{k+1},\psi^{k+1}, \mu^{k+1}, \kappa^{k+1} \in  \mathcal{C}_{\overline{m} \times \overline{n}}$ periodic such that
	\begin{eqnarray}
\psi^{k+1} - \psi^k &=& s M \Delta_h \mu^{k+1} - s\beta \psi^{k+1} , 
	\label{scheme1_1}
	\\
\phi^{k+1} - \phi^k &=& s\psi^{k+1},
	\label{scheme1_2}
	\\
\mu^{k+1} &=&  \delta_{\phi} F_c(\phi^{k+1}) - \delta_{\phi} F_e(\phi^{k})
	\nonumber
	\\
&=& \left(\phi^{k+1}\right)^3 + \alpha \phi^{k+1} + 2 \Delta_h \phi^{k} + \Delta_h \kappa^{k+1},
	\label{scheme1_3}
	\\
\kappa^{k+1} &:=& \Delta_h \phi^{k+1} ,
	\label{scheme1_4}
	\end{eqnarray}
where $s>0$ is the time step size, which we will assume is unchanging, and $\psi^0 \equiv 0$.  Notice that the part coming from the convex energy, $F_c$, is treated implicitly, and the part coming from $F_e$ is treated explicitly.  In this case, such a treatment leads to a nonlinear scheme, though the nonlinearity is quite mild.  This scheme can be simplified by using (\ref{scheme1_2}) to eliminate the unknown $\psi^{k+1}$ from (\ref{scheme1_1}).  Therefore, an equivalent form is
	\begin{eqnarray}
\phi^{k+1} -\phi^k &=& \frac{sM}{\beta+\frac{1}{s}} \Delta_h \mu^{k+1} + \frac{1}{\beta+\frac{1}{s}}\psi^k ,
	\label{discrete_1_a}
	\\
\psi^{k+1} &=& \frac{\phi^{k+1} - \phi^k}{s} .
	\label{discrete_1_b}
	\end{eqnarray}
Note that Eqs.~(\ref{discrete_1_a}) and (\ref{discrete_1_b}) decouple, and $\psi^{k+1}$ can be obtained after $\phi^{k+1}$ is known.  The three equations  (\ref{scheme1_3}),  (\ref{scheme1_4}), and (\ref{discrete_1_a}) are solved  simultaneously using the non-linear multigrid method that is described in App.~\ref{subapp-multigrid-first}. Then $\psi^{k+1}$ is updated via (\ref{discrete_1_b}).

	\subsection{Second-Order Convex Splitting Scheme}
	\label{subsec-2nd-order-scheme}

Following our work in~\cite{hu09,shen11} we propose the following second-order convex splitting scheme: given $\phi^k,\, \psi^k,\, \phi^{k-1} \in \mathcal{C}_{\overline{m} \times \overline{n}}$ periodic, find $ \phi^{k+1},\, \psi^{k+1}, \, \mu^{k+\frac{1}{2}}, \, \kappa^{k+\frac{1}{2}} \in  \mathcal{C}_{\overline{m} \times \overline{n}}$ periodic, such that
	\begin{eqnarray}
\psi^{k+1} -\psi^k &=& s M \Delta_h \mu^{k+\frac{1}{2}} - s\beta \psi^{k+\frac{1}{2}} ,
	\label{mpfc_scheme2_1}
	\\ 
\phi^{k+1} -\phi^k &=& s\psi^{k+\frac{1}{2}} ,
	\label{mpfc_scheme2_2}
	\\
\mu^{k+\frac{1}{2}} &=&   \chi\left(\phi^{k+1},\phi^k\right)\phi^{k+\hf} + \alpha \phi^{k+\frac{1}{2}} + 2\Delta_h \tilde{\phi}^{k+\hf} + \Delta_h \kappa^{k+\frac{1}{2}},
	\label{mpfc_scheme2_3}
	\\
\kappa^{k+\frac{1}{2}} &:=& \Delta_h \phi^{k+\frac{1}{2}} ,
	\label{mpfc_scheme2_4}
	\end{eqnarray}
where $s>0$ is the time step -- which, again, we take to be unchanging for simplicity -- and 
	\begin{displaymath}
\chi(\phi,\psi) :=  \frac{\phi^2+\psi^2}{2}, \quad \psi^{k+\frac{1}{2}} := \frac{\psi^{k+1} + \psi^k }{2} ,\quad \phi^{k+\frac{1}{2}} := \frac{\phi^{k+1} + \phi^k }{2} , \quad \tilde{\phi}^{k+\frac{1}{2}} := \frac{3 \phi^k - \phi^{k-1}}{2} .
	\end{displaymath}
We enforce the initial conditions $\psi^0 \equiv 0$ and $\phi^{-1}\equiv\phi^0$.  Note that the approximation $\phi^{-1}\equiv\phi^0$ will result in a first order (in time) local truncation error at the first step of the scheme.  This does not, however spoil the global second-order convergence of the method, as our tests later show.

Here again, we respect the convexity splitting noted earlier.  The convex contribution to the chemical potential is treated implicitly, but now using a second-order secant type approach as in~\cite{hu09,shen11}.  The concave contribution is treated explicitly using a second-order extrapolation.  This leads to a nonlinear, two-step method.   As before, we can decouple some of the equations in the scheme.  In particular, eliminating $\psi^{k+1}$ from (\ref{mpfc_scheme2_1}) we obtain
	\begin{eqnarray}
\phi^{k+1} -\phi^k &=& \frac{sM}{\beta+ \frac{2}{s}} \Delta_h \mu^{k+\frac{1}{2}} + \frac{2}{\beta+ \frac{2}{s}}\psi^k ,
	\label{discrete_2_a}
	\\
\psi^{k+1} &=& \frac{2}{s}\left( \phi^{k+1} -\phi^k\right) - \psi^k ,
	\label{discrete_2_b} 
	\end{eqnarray}
so that $\psi^{k+1}$ may be calculated after $\phi^{k+1}$ is known.  The three equations  (\ref{mpfc_scheme2_3}), (\ref{mpfc_scheme2_4}), and (\ref{discrete_2_a})  are solved  simultaneously using the non-linear multigrid method that is described in App.~\ref{subapp-multigrid-second}.  Then $\psi^{k+1}$ is updated via (\ref{discrete_2_b}).

	\subsection{Mass Conservation and Unconditional Unique Solvability}

The above schemes are discretely mass conserving and uniquely solvable for any time step $s>0$. These facts were proven for the first-order convex splitting scheme (\ref{scheme1_3}) -- (\ref{discrete_1_b}) in \cite{wang11}. We now provide the proofs for the second-order convex splitting scheme (\ref{mpfc_scheme2_3}) -- (\ref{discrete_2_b}).
	\begin{thm}
(Mass Conservation): The first-order convex splitting scheme  (\ref{scheme1_3}) -- (\ref{discrete_1_b}) and the second-order convex-splitting scheme (\ref{mpfc_scheme2_3}) -- (\ref{discrete_2_b}) for the MPFC equations are mass conserving for any time step $s>0$ provided solutions exists to the respective schemes.
	\end{thm}
	\begin{proof}
The details for the first-order scheme are contained in~\cite[Thm.~4.2]{wang11}. Suppose that $\left(\phi^{k+1},\psi^{k+1},\mu^{k+\hf},\kappa^{k+\hf}\right)$ is a solution to the second order scheme (\ref{mpfc_scheme2_3}) -- (\ref{discrete_2_b}).  Summing Eq.~(\ref{mpfc_scheme2_1}) and using summation-by-parts, specifically Prop.~\ref{green1stthm-2d},  we have 
	\begin{eqnarray}
\ciptwo{\psi^{k+1} - \psi^k}{{\bf 1}}  &=& s M \ciptwo{\Delta_h \mu^{k+\hf}}{{\bf 1}} - s\beta \ciptwo{ \frac{\psi^{k+1} + \psi^{k}}{2} }{{\bf 1}}
	\nonumber
	\\
&=& -s M \eipew{D_x \mu^{k+\hf}}{ D_x {\bf 1}} - s M \eipns{D_y \mu^{k+\hf}}{D_y {\bf 1}} -  s\beta\ciptwo{\frac{\psi^{k+1} + \psi^{k}}{2}}{{\bf 1}}
	\nonumber
	\\
&=&  - s \beta \ciptwo{ \frac{\psi^{k+1} + \psi^{k}}{2}}{{\bf 1}} .
	\end{eqnarray}
This gives the relation
	\begin{equation}
\ciptwo{\psi^{k+1}}{{\bf 1}}  =   \frac{1- \frac{\beta s}{2}}{1+ \frac{\beta s}{2}} \ciptwo{ \psi^{k}}{ {\bf 1}} .
	\end{equation}
This fact, together with the initial condition $\psi^0 \equiv 0$, ensures that $\ciptwo{\psi^{k}}{{\bf 1}} =0$  for all $k \geq 0$.  Now, observe that from (\ref{mpfc_scheme2_2})
	\begin{equation}
\ciptwo{\phi^{k+1} - \phi^k }{ {\bf 1} } = 0 \quad \mbox{if and only if} \quad  \ciptwo{\frac{\psi^{k+1} + \psi^{k}}{2}}{{\bf 1}} = 0 ,
	\label{mass_cons2}
	\end{equation}
and the result follows: $\ciptwo{\phi^{k+1}}{{\bf 1}} = \ciptwo{\phi^k}{{\bf 1}}$.
	\end{proof}

At this point it is worth emphasizing that for the simplified initial condition $\psi^0 \equiv 0$ we have $\ciptwo{\psi^k}{{\bf 1}} = 0$ for all $k>0$.  In other words, $\psi^k$ is a mean zero, periodic function for all $k>0$.  It would be enough for our purposes to enforce a more general initial condition, namely, $\ciptwo{\psi^0}{{\bf 1}} = 0$, though we do not pursue this level of generality here.

	\begin{thm}
(Unconditional Unique Solvability): The first-order convex splitting scheme  (\ref{scheme1_3}) -- (\ref{discrete_1_b}) and the second-order convex splitting scheme (\ref{mpfc_scheme2_3}) -- (\ref{discrete_2_b}) for the MPFC equations are uniquely solvable for any time step size $s>0$.
	\end{thm}

	\begin{proof}
The unconditional unique solvability for the first-order convex splitting scheme was proved in~\cite[Thm.~4.3]{wang11}, and we skip the details for that case here.  The proof for the second-order convex splitting scheme  (\ref{mpfc_scheme2_3}) -- (\ref{discrete_2_b}) closely follows our techniques from~\cite{wang11} for the first-order MPFC scheme and~\cite{hu09} for the second-order PFC scheme.  Specifically, we first construct the functional 
	\begin{equation}
G(\phi) := \frac{1}{2}\nrm{\phi-\phi^k}_{-1}^2 - \frac{2}{\beta + \frac{2}{s}} h^2 \ciptwo{\phi-\overline{\phi}}{\psi^k}_{-1} + \frac{sM}{\beta + \frac{2}{s}} R (\phi) ,
	\end{equation} 
where $\overline{\phi}:= \frac{1}{m n}\ciptwo{\phi^k}{{\bf 1}}$, 
	\begin{equation}
R(\phi) := Q(\phi )  + h^2\ciptwo{ \phi}{\frac{\alpha}{2} \phi^k + 2\Delta_h\tilde{\phi}^{k+\hf} + \frac{1}{2}\Delta_h^2 \phi^k} ,
	\end{equation}
and
	\begin{equation}
Q(\phi) := \frac{h^2}{4} \ciptwo{ \frac{\phi^4}{4} + \frac{\phi^3}{3}\phi^k + \frac{\phi^2}{2} \left(\phi^k\right)^2 +\phi\left(\phi^k\right)^3}{{\bf 1}} + \frac{\alpha}{4} \nrm{\phi}_2^2 + \frac{1}{4} \nrm{\Delta_h \phi}_2^2 .
	\end{equation}
Following arguments in~\cite{hu09}, we see that $R( \phi )$ is strictly convex and
	\begin{equation}
\delta_{\phi}R\left(\phi^{k+1}\right) = \mu^{k+\frac{1}{2}} .
	\end{equation}
This, together with the properties of the ``$-1$" inner product and norm given in App.~\ref{app-summation-by-parts}, implies that $G(\phi)$ is strictly convex and coercive over the set of admissible functions, the hyperplane
	\begin{displaymath}
\mathcal{A} = \left\{ \phi\in\mathcal{C}_{\overline{m}\times \overline{n}} \  \middle| \ \ciptwo{\phi}{{\bf 1}} = \ciptwo{\phi^k}{{\bf 1}} \ \mbox{and} \ \phi \  \mbox{is periodic}\right\},
	\end{displaymath}	
and its unique minimizer $\phi^{k+1}\in\mathcal{A}$ satisfies the discrete Euler-Lagrange equation
	\begin{equation}
\delta_{\phi}G\left(\phi^{k+1}\right) = -\Delta_h^{-1} \left( \phi^{k+1} - \phi^k - \frac{2}{\beta+ \frac{2}{s}} \psi^k \right) + \frac{sM}{\beta+  \frac{2}{s}} \mu^{k+\frac{1}{2}} + C = 0,
	\end{equation}
where $C$ is a constant.  This is equivalent to 
	\begin{equation}
 \phi^{k+1} - \phi^k  = \frac{sM}{\beta + \frac{2}{s}} \Delta_h \mu^{k+\frac{1}{2}} + \frac{2}{\beta+\frac{2}{s}} \psi^k , 
	\end{equation}
which is just Eq.~(\ref{discrete_2_a}).  Thus minimizing the strictly convex functional $G(\phi)$ over the set of admissible functions $\mathcal{A}$ is the same as solving the second-order convex splitting  scheme (\ref{mpfc_scheme2_3}) -- (\ref{discrete_2_b}). This completes the proof.
	\end{proof}
	
	\subsection{Unconditional Energy Stability}
The following theorem was proven in \cite{wang11}.
	\begin{thm}
(Unconditional Energy Stability of the First-Order Scheme):  The first-order convex splitting scheme (\ref{scheme1_3}) -- (\ref{discrete_1_b}) is unconditionally (strongly) pseudo energy stable with respect to (\ref{discrete_energy_MPFC}), meaning that for any time step $s>0$,
	\begin{equation}
\mathcal{F}\left(\phi^{k+1},\psi^{k+1}\right) \leq \mathcal{F}\left(\phi^k,\psi^k\right) .
	\end{equation}
	\end{thm}
We now proceed to show a similar result for the new second-order convex splitting scheme.

	\begin{thm}
(Unconditional Energy Stability of the Second-Order Convex Splitting Scheme): The second-order convex splitting scheme  (\ref{mpfc_scheme2_3}) -- (\ref{discrete_2_b}) is unconditionally (strongly) energy stable with 
respect to the discrete energy 
	\begin{equation}
\widetilde{\mathcal{F}} \left(\phi^{k},\phi^{k-1},\psi^{k}\right) :=  \mathcal{F}\left(\phi^{k},\psi^{k}\right) + \frac{1}{2} \nrm{\nabla_h \left( \phi^{k} - \phi^{k-1}\right)}^2_2 .
	\label{energy_numer}
	\end{equation}
In other words, for any $s>0$ and any $h>0$,
	\begin{equation}
\widetilde{\mathcal{F}} \left(\phi^{k+1},\phi^{k},\psi^{k+1}\right) \leq\widetilde{\mathcal{F}} \left(\phi^{k},\phi^{k-1},\psi^{k}\right) .
	\end{equation}
More specifically, 
	\begin{equation}
\widetilde{\mathcal{F}} \left(\phi^{k+1},\phi^{k},\psi^{k+1}\right) +  s \frac{\beta}{M}  \nrm{\psi^{k+\frac{1}{2}}}^2_{-1} + \frac{s^4}{2}\nrm{\nabla_h \left(D^2_s \phi^k \right)}^2_2 =  \widetilde{\mathcal{F}} \left(\phi^{k},\phi^{k-1},\psi^{k}\right) , 
	\end{equation}
where
	\begin{equation}
\label{defn_delta_s}
D^2_s \phi^k := \frac{1}{s^2} \left( \phi^{k+1} -2 \phi^{k} +\phi^{k-1} \right) .
	\end{equation}
	\end{thm}

	\begin{proof}
Using the identity 
	\begin{eqnarray}
h^2\ciptwo{\phi^{k+1} - \phi^k}{2\Delta_h \tilde\phi^k} &= &\displaystyle  -\nrm{ \nabla_h \phi^{k+1}}_2^2 + \frac{1}{2} \nrm{ \nabla_h \left(\phi^{k+1} - \phi^k \right)}_2^2
	\nonumber
	\\
&& + \nrm{\nabla_h \phi^k}_2^2 - \frac{1}{2} \nrm{\nabla_h\left(\phi^k - \phi^{k-1}\right)}_2^2
	\nonumber
	\\
&& + \frac{1}{2} \nrm{\nabla_h \left( \phi^{k+1} - 2 \phi^k +\phi^{k-1}\right)}_2^2						                         
	\end{eqnarray}
and Eq. (\ref{mpfc_scheme2_3}) we obtain
	\begin{eqnarray}
h^2\ciptwo{s\psi^{k+\frac{1}{2}}}{\mu^{k+\frac{1}{2}}} &= &h^2\ciptwo{\phi^{k+1} - \phi^k}{\mu^{k+\frac{1}{2}}}
	\nonumber
	\\
&= & F\left(\phi^{k+1}\right) + \frac{1}{2} \nrm{\nabla_h\left(\phi^{k+1} - \phi^k\right)}^2_2
	\nonumber
	\\
&& - F\left(\phi^{k}\right)- \frac{1}{2} \nrm{\nabla_h\left(\phi^{k} - \phi^{k-1}\right)}^2_2 
	\nonumber
	\\
&& + \frac{1}{2} \nrm{\nabla_h \left( \phi^{k+1} - 2 \phi^k +\phi^{k-1} \right)}_2^2 .
	\label{estimate_1}
	\end{eqnarray}
Next, using Eq.~(\ref{mpfc_scheme2_1}) and properties of the norm $\nrm{\, \cdot\, }_{-1}$ (see App.~\ref{app-summation-by-parts}), we have
	\begin{eqnarray}
\frac{1}{2M}  \nrm{\psi^{k+1}}^2_{-1} -  \frac{1}{2M} \nrm{\psi^{k}}^2_{-1} &=& \frac{h^2}{M}\ciptwo{\psi^{k+\hf}}{ \psi^{k+1} - \psi^{k}}_{-1} 
	\nonumber
	\\
&=&  \frac{h^2}{M} \ciptwo{\psi^{k+\frac{1}{2}}}{ s\Delta_h \mu^{k+\frac{1}{2}} - s\beta  \psi^{k+\frac{1}{2}}}_{-1} 
	\nonumber
	\\
&=& - s\beta \nrm{ \psi^{k+\frac{1}{2}}}_{-1}^2 + s h^2 \ciptwo{\psi^{k+\frac{1}{2}}}{-\Delta_h^{-1} \left( \Delta_h \mu^{k+\frac{1}{2}} \right)}
	\nonumber
	\\
&=&- s \beta \nrm{\psi^{k+\frac{1}{2}}}_{-1}^2 - s h^2 \ciptwo{\psi^{k+\frac{1}{2}}}{\mu^{k+\frac{1}{2}}}.
	\label{estimate_2}														\end{eqnarray}
Adding Eq.~(\ref{estimate_1}) and Eq.~(\ref{estimate_2}) we have the result.
	\end{proof}

	\section{Numerical Results}
	\label{sec-numerical-results}

In this section we present some numerical results for the convex splitting schemes that illustrate the convergence rates and energy stability of the schemes.  We conclude the section by applying the model to a problem of crystallization in an undercooled melt. 

	\subsection{Convergence Tests}
	
In order to study the theoretical convergence of the two schemes we evolve the same initial data in time with increasingly finer grid spacing and compare the solutions. We choose a domain $\Omega = (0,32) \times (0,32) $ with an initial density field 
	\begin{eqnarray}
\phi (x,y,t=0) &=&  0.07 - 0.02 \cos \left( \frac{2\pi (x-12)}{32} \right) \sin \left( \frac{2\pi (y-1)}{32} \right) \nonumber
	\\
&& + 0.02 \cos^2 \left( \frac{\pi (x+10)}{32} \right) \cos^2 \left( \frac{\pi (y+3)}{32} \right)  \nonumber
	\\
&&  - 0.01 \sin^2 \left( \frac{4\pi x}{32} \right) \sin^2 \left( \frac{4\pi (y-6)}{32} \right) .
	\label{initial-data}
	\end{eqnarray}
The parameters are taken to be $M=1$, $\epsilon =0.025$ ($\alpha = 1 - \epsilon$) and $\beta =0.9$, and we evolve the system to final time $t_f = 10$.  For these initial data and parameters, the system evolves toward a nonuniform steady state.

The time step refinement paths are given by  $s=0.025 h^2$ for the first-order convex splitting scheme (\ref{scheme1_3}) -- (\ref{discrete_1_b}) and $s=0.05 h$ for the second-order convex splitting scheme  (\ref{mpfc_scheme2_3}) -- (\ref{discrete_2_b}).  Note that these refinement paths have nothing to do with CFL-type stability constraints.  The schemes are unconditionally stable.  The important point is that, using these paths, in both cases the global error is predicted to be $\mathcal{E}(t_f) = \mathcal{O}(h^2)$, and this will be verified in the tests.  The fine details of the convergence tests are the same as those described in~\cite{hu09,shen11}.  We skip the specifics and direct the interested reader to those references for a more complete discussion.  The essential idea is to compare solutions at successively finer resolutions, one representing the coarse resolution, $h = h_c$, and one the fine resolution, $h=h_f$. 

The mesh spacings are taken to be $h = \frac{32}{32}$, $\frac{32}{64}$, $\frac{32}{128}$, $\frac{32}{256}$, and $\frac{32}{512}$, each one half the size of the previous.  The results of the tests are given in Tabs.~\ref{tab1} and \ref{tab2} for the first and second order convex splitting schemes,  respectively.  The test data provide evidence that both schemes are second order accurate in space.  They suggest that the scheme (\ref{scheme1_3}) -- (\ref{discrete_1_b}) is first-order accurate in time, \emph{i.e.}, $\mathcal{E}_{1}(t_f) = \mathcal{O}(s)+\mathcal{O}(h^2)$, while (\ref{mpfc_scheme2_3}) -- (\ref{discrete_2_b}) is second order accurate in time, \emph{i.e.}, $\mathcal{E}_{2}(t_f) = \mathcal{O}(s^2)+\mathcal{O}(h^2)$, as expected.

	\subsection{Multigrid Efficiency Test}
	
We now present a test, similar to one from~\cite{wise10}, that shows the efficiency of our nonlinear multigrid solvers described in App.~\ref{app-nonlinear-multigrid}. Specifically, we fix the time step size at the relatively large value of $s = 10.0$ and vary the spatial step sizes: $h = \frac{32}{32}$, $\frac{32}{64}$, $\frac{32}{128}$, $\frac{32}{256}$, and $\frac{32}{512}$.  We use the initial data (\ref{initial-data}), and the parameters are the same as in the last test: $M=1$, $\epsilon =0.025$ ($\alpha = 1 - \epsilon$), $\beta =0.9$, $\Omega = (0,32)\times(0,32)$.  We advance the solution using our second-order convex splitting scheme (\ref{mpfc_scheme2_3}) -- (\ref{discrete_2_b}) by exactly 20 time steps, so that the final time is $t_f = 200$.  We use $\ell_{\max}=1$ in the definition of the smoothing operator (cf. App.~\ref{subapp-multigrid-second}).  In Fig.~\ref{fig1} we report the reduction of the norm of the residual for each v-cycle iteration and for each space step size $h$.  The reader will observe that, independent of $h$, the norm of the residual is decreased by a factor of about 5.5.  This indicates that the norm of the error -- not the discretization error, but the algebraic iteration error of our implicit scheme -- is reduced by roughly the same amount.  This gives strong evidence that our multigrid solvers have optimal, or nearly optimal, complexity. The results of the corresponding test for the first-order scheme (\ref{scheme1_3}) -- (\ref{discrete_1_b}) are expected to be similar; we skip this test for the sake of brevity.

	\subsection{Energy Stability Test}
	\label{subsec-energy-stability}

In this section, we demonstrate the energy stability of our second-order convex-splitting scheme (\ref{mpfc_scheme2_3}) -- (\ref{discrete_2_b}).  We again use the initial data (\ref{initial-data}). The physical parameters are $M=1$, $\epsilon =0.025$ ($\alpha = 1 - \epsilon$), $\Omega = (0,32)\times(0,32)$, and $\beta = 0.01$.  These are the same as in the last couple of tests, except for the smaller value of $\beta$, which changes the time scaling and also emphasizes the wave nature of the equation more than in the other tests herein.  We use $h = \frac{32}{128}$ and $s = 0.1$ and advance the numerical solution to $t_f = 100$.  We plot the calculated discrete energies $\mathcal{F}$, defined in Eq. (\ref{discrete_energy_MPFC}); $\widetilde{\mathcal{F}}$, defined in Eq. (\ref{energy_numer}); and $F$, defined in Eq. (\ref{discrete_energy_PFC}), against time in Fig.~\ref{fig2}.  As pointed out earlier, the energy $F$ may be increasing in time; this phenomenon is observed here.  This test is somewhat contrived, since the value of $\beta$ may be smaller than what is typically used in practice.  If $\beta$ is large, in fact $F$ is often observed to be non-increasing.  However, this shows the point that $F$ is the incorrect energy for this dissipative system.  The only energy that is guaranteed to be non-increasing is $\widetilde{\mathcal{F}}$; this property is observed here.  But, note that $\mathcal{F}$ is nearly identical to $\widetilde{\mathcal{F}}$.  This phenomenon is expected when $s$ is sufficiently small.

	\subsection{Analysis of Effective Time Step Sizes}
	
We now proceed to study the effect of the choice of time step $s$ on the predicted crystal dynamics following similar tests in~\cite{cheng08, hu09}. We perform simulations on a domain of size $\Omega = (0,128)\times (0,128)$ and initial data $\phi_{i,j} = \bar{\phi} + \eta_{i,j}$ where $\bar{\phi}=0.07$ and $\eta_{i,j}$ is a uniformly distributed random number satisfying $|\eta_{i,j}| \leq 0.07$. The parameters are $M=1$, $\beta =0.9$, $\epsilon =0.025$ ($\alpha = 1-\epsilon$) and $h = 1$.  

To create a baseline, we first perform a simulation with the fully second order scheme (\ref{mpfc_scheme2_3}) -- (\ref{discrete_2_b}) with the relatively small time step size $s=s_1 =0.01$. The density fields at time $t=350, t=500, t=1500$ and $t=2500$ are shown in the first row of Fig.~\ref{fig3}.  Then we perform successive simulations with $s_2=1$, $s_3=10$ and $s_4=20$ (second, third, and fourth columns of Fig.~\ref{fig3}) again using the second-order convex splitting scheme (\ref{mpfc_scheme2_3}) -- (\ref{discrete_2_b}), with all other physical and numerical parameters being the same.  Following the procedure in~\cite{cheng08,hu09} we compare frames of the different simulations when the pseudo energy $\mathcal{F}$ (\ref{discrete_energy_MPFC}) values are the same, rather than when the times are the same.  Each column of Fig.~\ref{fig3} represents configurations of $\phi$ that have the same pseudo energy $\mathcal{F}$.  The scaled pseudo energy $\mathcal{F}$ is plotted for the $s=s_1=0.01$ case in Fig.~\ref{fig4}; it is observed to be non-increasing in time.

Now, we repeat the test using the first-order convex splitting scheme (\ref{scheme1_3}) -- (\ref{discrete_1_b}) and report the results in Fig.~\ref{fig5}.  But here, for the first row of Fig.~\ref{fig5}, we plot $\phi$ at the times for which the pseudo energy values match those for $\phi$ represented in the first row of Fig.~\ref{fig3}, calculated using the more accurate second-order scheme (\ref{mpfc_scheme2_3}) -- (\ref{discrete_2_b}).  Again, the $\phi$ configurations in each column of Fig.~\ref{fig5} have the same pseudo energy $\mathcal{F}$ values, which explains the discrepancies in the corresponding times.

Inspecting Figs.~\ref{fig3} and \ref{fig5}, the qualitative agreement of the density fields $\phi$ corresponding to the same pseudo energy values (plots in the same respective columns) is apparent.  However, the numerical times required to reach the specific pseudo energy levels can be wildly different.  For example, for the first-order convex splitting scheme for the $s_4$ case (Fig.~\ref{fig5} row 4) it takes $t=122640$ to reach the energy corresponding to the second order scheme at $t=2500$ for $s_1$, about 50 times longer. Meanwhile, the second order scheme reaches the same energy level for $s_4$ case at $t=5520$ indicating the higher time accuracy of the second order scheme.  The second-order convex splitting scheme is clearly more accurate, at least at a qualitative level, for this practical test.

Following~\cite{cheng08,hu09}, in order to study the quantitative difference in respective solutions (obtained with different time step sizes $s$) at the same final time $t_f$, we define the scaled difference
	\begin{equation}
D(s) = \frac{\nrm{\phi\left(\, \cdot,t_f;s_1\right) - \phi\left(\, \cdot\, ,t_f;s\right)}_2}{\nrm{\phi\left(\, \cdot\, ,t_f;s_1\right)}_2},
	\label{scaled-difference}
	\end{equation}
where we choose $t_f=350$.  The other parameters are the same as in the previous test: $M=1$, $\beta =0.9$, $\epsilon =0.025$ ($\alpha = 1-\epsilon$) and $h = 1$.  Figure~\ref{fig6} shows the log-log plot of the scaled difference obtained using the second order convex splitting scheme at time $t_f =350$ and the baseline time step size $s=s_1 =0.01$.  The results are quite similar to that observed for the PFC schemes in~\cite{cheng08,hu09}. The accuracy is observed to increase as the time step is decreased, while the difference (the error) tends to saturate for large time steps.

	\subsection{Growth of a Polycrystal from a Supercooled Liquid}

In this test, we apply our second-order scheme to solve a problem of polycrystallization in a supercooled liquid. We start our simulation with a constant density field $\bar{\phi}=0.285$ on a domain $\Omega = (0,804) \times (0,804)$. We place three random perturbations (small in spatial extent) along the bottom $y=0$ line of the domain as seeds for nucleation. The domain is discretized using a $2048\times 2048$ grid, \emph{i.e.}, $h=804/2048$.The parameters are set to be $M=1$, $\epsilon = 0.25$ ($\alpha = 1-\epsilon$), and $\beta =0.9$, and snapshots of the crystal microstructure are shown in Figs.~\ref{fig7a} -- \ref{fig7d}.

Here the boundary conditions are periodic in the horizontal direction and homogeneous Neumann on the top and bottom boundaries.  Our theory can be easily extended for these conditions.  This case study serves to show the applicability of the model to a physical problem; it also shows the relative flexibility of our multigrid/finite-difference framework over spectral and pseudo-spectral methods. In particular, it is a triviality to modify the nonlinear multigrid solver to adjust to the present mixed boundary conditions from fully periodic boundary conditions. This is not the case for spectral methods, where the entire method mat have to be redesigned for different boundary conditions.

Three different crystal grains are observed to nucleate from the seed sites and grow. These grains eventually grow large enough to form grain boundaries as they impinge upon one another. In Figs.~\ref{fig7b} -- \ref{fig7d}, the grain in the middle forms boundaries with the grains on the left and right.  However, the orientations of the two crystallites on the right and left extremes are so similar to one another that they relax without a definitive grain boundary, but form point defects (about $x=50$, $t=1000$ and $t = 2000$) along where the grains impinge. 

In order to study the effect of the damping term we repeat the simulation with the same initial data and parameters, with the exception that we take $\beta =10$ to obtain higher relative damping than in the last test. A single snapshot at time $t=3000$ for this new simulation is shown in Fig.~\ref{fig8}. The increased damping manifests itself in two significant ways compared to the previous case. The first is in the qualitative features of the grain boundaries, and the second is in the crystal growth rate.  

Some of the second can be explained by the change in the diffusion time scale with the modification of $\beta$. If we rescale the present case so that the diffusion time scale is the same as that of the last, we obtain
	\begin{equation}
\left(\frac{0.9}{10}\right)^2 \partial_{ss}\phi + 0.9 \partial_{s}\phi = \Delta\mu,
	\end{equation}
where $s = \frac{0.9}{10}t$.  Thus $t=3000$ corresponds to the synchronized diffusion time $s = 270$.  The earliest snapshot shown from the last simulation is at time $t = 500$ (Fig.~\ref{fig7a}).  Clearly for the high damping case the crystal is growing into its melt \emph{faster} (from the perspective of the synchronized diffusion time scale) than in the low damping case.

On the first point, the increased damping ($\beta =10$) of the propagation of waves allows the left grain and the middle grain to merge into one crystallite of a single grain with a single point defect located at the bottom (at $x=250$) unlike the case with $\beta =0.9$.  One should compare the microstructure with that of Fig.~\ref{fig7b} ($t = 1000$).   More quantitative differences between the low and high damping cases can be observed from the relative pseudo energies, which are plotted for both simulations in Fig.~\ref{fig9}.

	\subsection{Effect of Applied Strain on Coherent Crystal}
	\label{section_relaxation_test}

The key difference between the the PFC and MPFC models is the ability of the latter to capture the difference between the ``elastic" relaxation and ``diffusion" time scales~\cite{stefanovic06}.  To show the power of the MPFC model we perform a test that highlights this difference.  Specifically, we will study the relaxation of a strained crystal by tracking the displacement field with respect to the strain-free crystal configuration over time.

We take $\epsilon=0.6$ ($\alpha = 1-\epsilon$), $M =15^2$, $h=\frac{4\pi}{16 \sqrt{3}}$, $s=0.025 h^2$, and use the second-order convex splitting scheme for the tests.  We use two different values of $\beta$ in the tests.  The grid sizes are $m = 512$ (grid points in $x$ direction) and $n=400$ (grid points in $y$ direction).  A solid periodic crystal, 20 atomic layers thick, was formed in a liquid bath, as shown in Fig.~\ref{fig10}, in the coexistence region, \emph{i.e.}, where the solid and liquid coexist at equilibrium.  The coexistence solid and liquid densities for $\epsilon =0.6$ are $\phi_s = 0.395$ and $\phi_l=0.57$, respectively~\cite{elder04}. The solid region density field is chosen to be the one given by the single mode approximation to the equilibrium density~\cite{elder04}:
	\begin{equation}
\phi(x,y) = \phi_s + A \left[ \cos\left(q_t x\right) \cos \left ( \frac{q_t y}{\sqrt{3}} \right ) - \frac{1}{2} \cos \left(\frac{2q_t y}{\sqrt{3}} \right) \right] ,
	\label{single_mode_density}
	\end{equation}
where
	\begin{equation}
A=\frac{4}{5}\phi_s  + \frac{4}{15}\sqrt{ 15 \epsilon - 36\phi_s^2} ,   \quad  q_t = \frac{\sqrt{3}}{2}.
	\end{equation}

The bottom row of atoms in this crystal were pinned in this configuration and then annealed to equilibrium, leading to the strain-free microstructure shown in Fig.~\ref{fig9}.  Then a shear strain was applied by pinning the density field of the atoms along the top row and shifting the density field of the bottom row of atoms by 4\% of the horizontal lattice spacing, \emph{i.e.}, 4\% of $2\pi$, to the left.  The pinning of atoms and the application of strain is accomplished by following the technique of Stefanovic~\emph{et al.}~\cite{stefanovic06,stefanovic09}. The idea is to modify the continuous free energy, $E(\phi)$, to obtain
	\begin{equation}
\tilde{E}(\phi) := E(\phi) + \int_\Omega \mathcal{M}(x,y) \left( \phi(x,y) - \hat{\phi}(x,y) \right)^2  dx dy ,
	\end{equation}
where $\mathcal{M}\ge 0$ is a weight function and $\hat{\phi}$ is a prescribed crystal configuration.  We take $\mathcal{M} = 2$ in horizontal strips of vertical thickness $\pi$ (about half the lattice spacing) around the top and bottom of the crystal, and otherwise $\mathcal{M}$ is set to zero.  The function $\hat{\phi}$ is chosen similar to \eqref{single_mode_density} but modified slightly so as to effect the desired pinning and shear~\cite{stefanovic06,stefanovic09}. The incorporation of the modified energy into our second-order convex splitting scheme is straightforward. We simply replace $\mu^{k+\frac{1}{2}}$ with
	\begin{equation}
\tilde{\mu}^{k+\frac{1}{2}} = \mu^{k+\frac{1}{2}} + 2 \left(\phi^{k+\frac{1}{2}} - \hat{\phi} \right) \mathcal{M}.
	\end{equation}

The horizontal strain at an atom position is defined via $\varepsilon_{xx} := \frac{\Delta x}{2\pi}$, where $\Delta x$ is the horizontal displacement measured from the reference configuration. At steady state, we expect the strain to vary linearly from 4\% to 0\% along the height of the crystal. Figure~\ref{fig11} shows the time evolution of the strain field for two computational cases:  $\beta =0.2$ (triangles in Fig.~\ref{fig11}, low damping case) and $\beta = 9$ (squares in Fig.~\ref{fig11}, high damping case).   Note the oscillatory nature of the low damping MPFC case (triangles). In this case the crystal rapidly relaxes to the strained equilibrium state, represented by the straight line.  The high damping MPFC case behaves much more like the PFC model, as expected, with a slow relaxation to equilibrium, but without the oscillatory overshoot (even at long times, not shown) that is seen in the low damping case.

	\section{Concluding Remarks}
	\label{sec-conclusions}
	
In this paper we have presented two unconditionally (pseudo) energy stable finite difference methods for the MPFC model. Both are based on convex splittings of the physical energy.  One is first-order accurate in time and appeared in our previous works~\cite{wang10b,wang11}.  The other is second-order accurate in time; and the way we deal with the the chemical potential term in the scheme is similar to that in our second-order convex splitting scheme for the PFC model~\cite{hu09}.  What makes this a non-trivial extension of the previous method is the treatment of the second-order time derivative $\partial_{tt}\phi$.  The second-order approximation of this term relies on the introduction of a new variable $\psi := \partial_t \phi$.

We provided numerical evidence of the respective orders of accuracy and convergences of the two schemes, the efficiency of our nonlinear multigrid solvers, and the (pseudo) energy stability of our schemes.  In particular, we showed that, even though our schemes are nonlinear, they can be solved with optimal (or near optimal) complexity using nonlinear FAS multigrid solvers.  We demonstrated that the second order accurate scheme generally captures the time evolution with greater accuracy and efficiency than the first order scheme, which is of course what would be expected.  We concluded the paper with a couple of practical computational examples, one of solidification and one dealing with elastic relaxation in a strained crystal.

We take this opportunity to emphasize the significance of the second-order convex splitting scheme and attempt to put the work in some broader context.  The novelty of our second-order convex splitting scheme is not that it enjoys an unconditional pseudo energy stability (ES) property and not that it enjoys an unconditionally unique solvability (US) property but that it enjoys \emph{both properties simultaneously}.  One can always construct a scheme that satisfies one or other of the two individually.  For example, fully  explicit schemes are usually always unconditionally uniquely solvable (trivially).  And one can always construct a second-order energy stable scheme using a secant-type (or Crank-Nicolson-like) approach.  Here, roughly speaking, given an energy $E$, one takes
	\begin{equation}
\mu^{k+\hf} := \frac{E\left(\phi^{k+1}\right) - E\left(\phi^k\right)}{\phi^{k+1} - \phi^k} ,
	\end{equation}
as in~\cite{du91}. (In fact, this is the idea we use to deal with the convex portion of the energy (and the chemical potential) in our own second-order schemes, here and in~\cite{hu09}.)  While, this will lead to an unconditionally energy stable scheme, the question of solvability is much harder.  Typically, unconditional solvability must be sacrificed.

As we have pointed out, our scheme enjoys both properties ES and US.  But this comes at a price.  Namely, we use a multistep method, and it is more difficult in this setting to employ adaptive time stepping.  What would be ideal is a one-step, second-order scheme enjoying both properties ES and US. In this case time adaptivity could be more straightforward.  Regardless of whether it is using a new second-order scheme or the present one, adaptive time and space stepping strategies constitute an important future  research direction.
 
	\section*{Acknowledgements}

JL and AB acknowledge partial support from NSF-CHE 1035218, JL also gratefully acknowledges partial support from NSF-DMS 0914720 (as does PZ), NSF-DMS 0915128, and NSF-DMR 1105409. CW acknowledges partial support from NSF-DMS 1115420, NSF-DCNS 0959382 and AFOSR-10418149. SMW acknowledges partial support from NSF-DMS 1115390.

	\appendix

	\renewcommand{\theequation}{A.\arabic{equation}}

	\section{Finite Difference Discretization on a Staggered Grid}
	\label{app-summation-by-parts}
      
For simplicity, we assume that $\Omega = (0,L_x)\times(0,L_y)$.  The framework that we describe has a straightforward extension to three space dimensions. Here we use the notation and results for cell-centered functions from~\cite{wise10,wise09}, see also~\cite{hu09,wang11}.  The reader is directed to those references for more complete details.  We begin with definitions of grid functions and difference operators needed for our discretization of two-dimensional space.  Let $\Omega = (0,L_x)\times(0,L_y)$, with $L_x = m\cdot h$ and $L_y = n\cdot h$, where $m$ and $n$ are positive integers and $h>0$ is the spatial step size. Define $p_r := (r-\hf)\cdot h$, where $r$ takes on integer and half-integer values.   For any positive integer $\ell$, define $E_\ell = \left\{ p_r \ \middle|\ r=\frac{1}{2},\ldots, \ell+\frac{1}{2}\right\}$, $C_\ell = \left\{p_r \ \middle|\ r=1,\ldots, \ell\right\}$,  $C_{\overline{\ell}} = \left\{ p_r\cdot h\ \middle|\ r=0,\ldots, \ell+1\right\}$. Define the function spaces
	\begin{eqnarray}
{\mathcal C}_{m\times n} &=& \left\{\phi: C_m\times C_n \rightarrow \mathbb{R} \right\},\  {\mathcal C}_{\overline{m}\times\overline{n}} = \left\{\phi: C_{\overline{m}}\times C_{\overline{n}}\rightarrow \mathbb{R} \right\}  ,
    \\
{\mathcal C}_{\overline{m}\times n} &=& \left\{\phi: C_{\overline{m}}\times C_n \rightarrow \mathbb{R} \right\},\ {\mathcal C}_{m\times\overline{n}} = \left\{\phi: C_m\times C_{\overline{n}} \rightarrow \mathbb{R} \right\}  ,
	\\
{\mathcal E}^{\rm ew}_{m\times n} &=& \left\{u: E_m\times C_n \rightarrow\mathbb{R} \right\},\ {\mathcal E}^{\rm ns}_{m\times n} = \left\{v: C_m\times E_n \rightarrow\mathbb{R}  \right\}	 ,
	\\
{\mathcal E}^{\rm ew}_{m\times \overline{n}} &=& \left\{u: E_m\times C_{\overline{n}} \rightarrow\mathbb{R} \right\},\ {\mathcal E}^{\rm ns}_{\overline{m}\times n} = \left\{v: C_{\overline{m}}\times E_n \rightarrow\mathbb{R}  \right\}	 .
  \end{eqnarray}
We use the notation $\phi_{i,j} := \phi\left(p_i,p_j\right)$ for \emph{cell-centered} functions, those in the spaces ${\mathcal C}_{m\times n}$, ${\mathcal C}_{\overline{m}\times n}$, ${\mathcal C}_{m\times\overline{n}}$, or ${\mathcal C}_{\overline{m}\times\overline{n}}$. In component form \emph{east-west edge-centered} functions, those in the spaces ${\mathcal E}^{\rm ew}_{m\times n}$ or ${\mathcal E}^{\rm ew}_{m\times \overline{n}}$, are identified via $u_{i+\hf,j}:=u(p_{i+\hf},p_j)$.  In component form \emph{north-south edge-centered} functions, those in the spaces ${\mathcal E}^{\rm ns}_{m\times n}$, or ${\mathcal E}^{\rm ns}_{\overline{m}\times n}$, are identified via $u_{i+\hf,j}:=u(p_{i+\hf},p_j)$.  The functions of ${\mathcal V}_{m\times n}$ are called \emph{vertex-centered} functions.

We need the weighted 2D grid inner-products $\ciptwo{\, \cdot \,}{\, \cdot \,}$, $\eipew{\, \cdot \,}{\, \cdot \,}$, $\eipns{\, \cdot \,}{\, \cdot \,}$ that are defined in~\cite{wise10,wise09}.  In addition to these, we also need the following one-dimensional inner-products:
	\begin{equation}
\cip{f_{\star,j+\hf}}{g_{\star,j+\hf}} = \sum_{i=1}^m f_{i,j+\hf}g_{i,j+\hf} , \quad \cip{f_{i+\hf,\star}}{g_{i+\hf,\star}} = \sum_{j=1}^n f_{i+\hf,j}g_{i+\hf,j}  ,	
	\end{equation}
where the first is defined for $f,\, g\in{\mathcal E}^{\rm ns}_{m\times n}$, and the second for $f,\, g\in{\mathcal E}^{\rm ew}_{m\times n}$.

The reader is referred to~\cite{wise10,wise09} for the precise definitions of the edge-to-center difference operators  $d_x : {\mathcal E}_{m\times n}^{\rm ew}\rightarrow{\mathcal C}_{m\times n}$ and $d_y : {\mathcal E}_{m\times n}^{\rm ns}\rightarrow{\mathcal C}_{m\times n}$; the $x-$dimension center-to-edge average and difference operators, respectively, $A_x,\, D_x: {\mathcal C}_{\overline{m}\times n}\rightarrow{\mathcal E}_{m\times n}^{\rm ew}$;  the $y-$dimension center-to-edge average and difference operators, respectively, $A_y,\, D_y: {\mathcal C}_{m\times \overline{n}}\rightarrow{\mathcal E}_{m\times n}^{\rm ns}$; and the standard 2D discrete Laplacian, $\Dh : {\mathcal C}_{\overline{m}\times\overline{n}}\rightarrow{\mathcal C}_{m\times n}$.

In this paper we use grid functions satisfying either periodic or homogeneous Neumann boundary conditions, or some appropriate combination of these.  Specifically, we shall say the cell-centered function $\phi\in {\mathcal C}_{\overline{m}\times\overline{n}}$ satisfies homogeneous Neumann boundary conditions if and only if
	\begin{eqnarray}
\phi_{m+1,  j} = \phi_{m,j}, \quad \phi_{  0,  j} &=& \phi_{1,j}, \quad j = 1,\ldots,n \ ,
	\label{h-n-bcs-1}
	\\
\phi_{  i,n+1} = \phi_{i,n}, \quad \phi_{  i,  0} &=& \phi_{i,1}, \quad i = 0,\ldots,m+1 \ .
	\label{h-n-bcs-2}
	\end{eqnarray}
We say that the cell-centered function $\phi\in {\mathcal C}_{\overline{m}\times\overline{n}}$ satisfies periodic boundary conditions if and only if
	\begin{eqnarray}
\phi_{m+1,  j} = \phi_{1,j}, \quad \phi_{0,  j} &=& \phi_{m,j}, \quad j = 1,\ldots,n \ ,
	\label{per-bcs-1}
	\\
\phi_{i,n+1} = \phi_{i,1}, \quad \phi_{i,0} &=& \phi_{i,n}, \quad i = 0,\ldots,m+1 \ .
	\label{per-bcs-2}
	\end{eqnarray}

We will use the grid function norms defined in~\cite{wise10,wise09}.  The reader is referred to those works for the precise definitions of  $\nrm{\, \cdot \,}_2$, $\nrm{\, \cdot \,}_\infty$, $\nrm{\, \cdot \,}_p$ ($1\le p < \infty$), $\nrm{\, \cdot \,}_{0,2}$, $\nrm{\, \cdot \,}_{1,2}$, and $\nrm{\phi}_{2,2}$.

Using the definitions given in this Appendix and in~\cite{wise10,wise09}, we obtain the following summation-by-parts formulae:
	\begin{prop}
	\label{sbp-2D-edge}
{\em (summation-by-parts)} If $\phi\in{\mathcal C}_{\overline{m}\times n} \cup{\mathcal C}_{\overline{m}\times\overline{n}}$ and $f\in{\mathcal E}_{m\times n}^{\rm ew}$ then
	\begin{eqnarray}
h^2\, \eipew{D_x \phi}{f} &=& -h^2\, \ciptwo{\phi}{d_x f}
	\nonumber
	\\
&&-h\, \cip{A_x\phi_{\hf,\star}}{f_{\hf,\star}}+h\, \cip{A_x\phi_{m+\hf,\star}}{f_{m+\hf,\star}} \ ,
	\label{sbp-c-ew}
	\end{eqnarray}
and if $\phi\in{\mathcal C}_{m\times\overline{n}} \cup {\mathcal C}_{\overline{m}\times\overline{n}}$ and $f\in{\mathcal E}_{m\times n}^{\rm ns}$ then
	\begin{eqnarray}
h^2\, \eipns{D_y\phi}{f} &=& -h^2\, \ciptwo{\phi}{d_y f} 
	\nonumber
	\\
&&-h\, \cip{A_y\phi_{\star,\hf}}{f_{\star,\hf}}+h\, \cip{A_y\phi_{\star,n+\hf}}{f_{\star,n+\hf}} \ .
	\label{sbp-c-ns}
	\end{eqnarray}
	\end{prop}

	\begin{prop}
	\label{green1stthm-2d}
\emph{(discrete Green's first identity)} Let $\phi,\, \psi\in {\mathcal C}_{\overline{m}\times\overline{n}}$.  Then
	\begin{eqnarray}
h^2\, \eipew{D_x\phi}{D_x\psi} &+&h^2\, \eipns{D_y\phi}{D_y\psi}
	\nonumber
	\\
&=& -h^2\, \ciptwo{\phi}{\Dh\psi}
	\nonumber
	\\
& & - h\, \cip{A_x\phi_{\hf,\star}}{D_x\psi_{  \hf,\star}} + h\, \cip{A_x\phi_{m+\hf,\star}}{D_x \psi_{m+\hf,\star}}
	\nonumber
	\\
& & - h\, \cip{A_y\phi_{\star,  \hf}}{D_y\psi_{\star,  \hf}} + h\, \cip{A_y \phi_{\star,n+\hf}}{D_y\psi_{\star,n+\hf}} \ . 
	\nonumber
	\\
& &
	\end{eqnarray}
	\end{prop}

	\begin{prop}\label{green2ndthm-2d}
	\emph{(discrete Green's second identity)} Let $\phi,\, \psi\in {\mathcal C}_{\overline{m}\times\overline{n}}$.  Then
	\begin{eqnarray}
h^2\, \ciptwo{\phi}{\Dh\psi} &=& h^2\, \ciptwo{\Dh\phi}{\psi}
	\nonumber
	\\
& & + h\, \cip{A_x \phi_{m+\hf,\star}}{D_x \psi_{m+\hf,\star}} - h\, \cip{D_x \phi_{m+\hf,\star}}{A_x \psi_{m+\hf,\star}}
	\nonumber
	\\
& & - h\, \cip{A_x \phi_{  \hf,\star}}{D_x \psi_{  \hf,\star}} + h\, \cip{D_x \phi_{  \hf,\star}}{A_x \psi_{  \hf,\star}}
	\nonumber
	\\
& & + h\, \cip{A_y \phi_{\star,n+\hf}}{D_y \psi_{\star,n+\hf}} - h\, \cip{D_y \phi_{\star,n+\hf}}{A_y \psi_{\star,n+\hf}}
	\nonumber
	\\
	& & - h\, \cip{A_y \phi_{\star,  \hf}}{D_y \psi_{\star,  \hf}} + h\, \cip{D_y \phi_{\star,  \hf}}{A_y \psi_{\star,  \hf}} \ .
	\end{eqnarray} 
	\end{prop}
All of the boundary sums above vanish when the boundary conditions are periodic or homogeneous Neumann.

We need the definition of the following space in the analysis of our scheme:
	\begin{equation}
	\label{mean-zero-space}
H := \left\{\phi\in{\mathcal C}_{m\times n} | \ciptwo{\phi}{{\bf 1}}=0 \right\}.
	\end{equation}
We  equip this space with the bilinear form
	\begin{equation}
\ciptwo{\phi_1}{\phi_2}_{-1} := \eipew{D_x \psi_1}{D_x \psi_2}  + \eipns{D_y \psi_1}{D_y \psi_2} ,
	\end{equation}
for any $\phi_1,\, \phi_2\in H$, where $\psi_i\in{\mathcal C}_{\overline{m}\times\overline{n}}$ is the unique solution to 
	\begin{equation}
-\Dh \psi_i  = \phi_i,\quad \psi_i\mbox{ periodic},\quad \ciptwo{\psi_i}{{\bf 1}} = 0.
	\end{equation}
The proof of the following can be found in~\cite{wang11}.
	\begin{prop}
	\label{h-l-inner-product}
$\ciptwo{\phi_1}{\phi_2}_{-1}$ is an inner product on the space $H$.  Moreover,
	\begin{equation}
\ciptwo{\phi_1}{\phi_2}_{-1} = -\ciptwo{\phi_1}{\Dh^{-1} \phi_2 } = -\ciptwo{\Dh^{-1} \phi_1 }{\phi_2}.
	\end{equation}
Thus
	\begin{equation}
\nrm{\phi}_{-1} :=\sqrt{h^2\,\ciptwo{\phi}{\phi}_{-1}}
	\label{H-norm}
	\end{equation}
defines a norm on $H$.
	\end{prop}

	\section{Nonlinear Multigrid Solvers}
	\label{app-nonlinear-multigrid}

	\renewcommand{\theequation}{B.\arabic{equation}}

In this Appendix we provide the details of the nonlinear multigrid algorithms that we use for time stepping the semi-implicit numerical schemes proposed in previous sections.  We note that the first-order convex splitting scheme was first presented in \cite{wang11}, but no implementation strategy or simulations were presented.

	\subsection{Multigrid Method for the First-Order Convex Splitting Scheme}
	\label{subapp-multigrid-first}
	
For the first-order convex splitting scheme (\ref{scheme1_3}) -- (\ref{discrete_1_b}), given $\phi^k,\, \psi^k \in \mathcal{C}_{\overline{m} \times \overline{n}}$, we need to find the grid functions  $\phi,\, \psi,\, \mu,\, \kappa \in \mathcal{C}_{\overline{m} \times \overline{n}}$ satisfying
	\begin{eqnarray}
\phi_{ij} - \phi^k_{ij} &=& \left(  \frac{sM}{\beta +\frac{1}{s}} \right) \Delta_h \mu_{ij} + \frac{1}{\beta+\frac{1}{s}} \psi^{k}_{ij},
	\label{NA1_1}
	\\
\mu_{ij} &=& \phi_{ij}^3 + \alpha\phi_{ij} + 2 \Delta_h \phi^k_{ij} + \Delta_h \kappa_{ij},
	\label{NA1_2}
	\\
\kappa_{ij} &=& \Delta_h \phi_{ij},
	\label{NA1_3}
	\\
\psi_{ij} &=& \frac{\phi_{ij}- \phi^{k}_{ij}}{s},
	\label{NA1_4}
	\end{eqnarray}
where we have dropped the superscript $k+1$.  We solve this system by observing that the first three equations decouple from the fourth. Thus, we solve for $\phi,\,  \mu,\, \kappa \in \mathcal{C}_{\overline{m} \times \overline{n}}$ using (\ref{NA1_1}) -- (\ref{NA1_3}) and then update $\psi \in \mathcal{C}_{\overline{m} \times \overline{n}}$ using (\ref{NA1_4}).

Equations  (\ref{NA1_1}) -- (\ref{NA1_3}) are nonlinear and can be solved using a nonlinear multigrid algorithm similar to the one proposed in~\cite{hu09}.  See also~\cite{baskaran10,kim03,wise10}.  The first task it to identify a nonlinear operator $\mathbf{N}$ and a source term $\mathbf{S}$, such that $\mathbf{N}=\mathbf{S}$ is equivalent to (\ref{NA1_1}) -- (\ref{NA1_3}).  Let $\mathbf{u} =  ( \phi,\mu,\kappa )^T$.  The $3 \times m \times n$ nonlinear operator $\mathbf{N}({\bf u}) = (N^{(1)}({\bf u}),N^{(2)}({\bf u}),N^{(3)}({\bf u}))^T$ is defined via
	\begin{eqnarray}
N^{(1)}_{ij} ({\bf u}) &:=& \phi_{ij} - \left(  \frac{sM}{\beta+\frac{1}{s}} \right) \Delta_h \mu_{ij} ,
	\\
N^{(2)}_{ij} ({\bf u}) &:=& \mu_{ij} - \phi_{ij}^3 - \alpha\phi_{ij} - \Delta_h \kappa_{ij} ,
	\\
N^{(3)}_{ij} ({\bf u}) &:=& \kappa_{ij} - \Delta_h \phi_{ij} .
	\end{eqnarray}
The $3 \times m \times n$ source $\mathbf{S} = \left(S^{(1)},S^{(2)},S^{(3)}\right)^T$ is defined via
	\begin{equation}
S^{(1)}_{ij}:=\phi^k_{ij} +\frac{1}{\beta+\frac{1}{s}} \psi^{k}_{ij}, \quad S^{(2)}_{ij} := 2\Delta_h \phi^k_{ij}, \quad S^{(3)}_{ij} :=0 .
	\end{equation}
Clearly the system (\ref{NA1_1}) -- (\ref{NA1_3}) is equivalent to $\mathbf{N}({\bf u}) = \mathbf{S}$.

The system $\mathbf{N}({\bf u}) = \mathbf{S}$ can be efficiently solved using a nonlinear FAS multigrid method. Extensive literature is available and we refer the reader to \cite{trottenberg01} for a detailed review.  Since we are using a standard FAS V-cycle approach, as in~\cite{hu09,kim03,wise10} we only provide the details concerning how the nonlinear smoothing operation is performed.  The other components of the algorithm can be constructed by the interested reader by consulting~\cite{trottenberg01}.

For the smoothing operator we use a nonlinear red-black Gauss-Siedel method.  Let $\ell$ denote the smoothing iteration.  By $\ell_{\max}$ we mean the maximum number of smoothing iterations.   We point out that the nonlinearity in the operator $\mathbf{N}$ comes only from the cubic term $\phi_{ij}^3$.  This can be handled with a local linearization.  We commonly use either a local Newton-type linearization or a simpler local Picard-type linearization, respectively,
	\begin{equation}
\left(\phi^{\ell+1}_{ij}\right)^3 \approx \left(\phi^{\ell}_{ij}\right)^3 + 3 \left(\phi^{\ell}_{ij}\right)^2 \left(\phi^{\ell+1}_{ij}-\phi^{\ell}_{ij}\right) \quad \mbox{or}\quad \left(\phi^{\ell+1}_{ij}\right)^3 \approx  \left(\phi^{\ell}_{ij}\right)^2 \phi^{\ell+1}_{ij}.
	\end{equation}
To simplify the description of the smoother, we write the details assuming a lexicographic ordering is followed and we employ the Picard linearization.  To get a proper description of the implementation of the red-black ordered Gauss-Seidel method, which is what we use in our codes, we refer the reader to~\cite{trottenberg01}.  The smoothing scheme can thus be written as
	\begin{eqnarray}
\phi_{i,j}^{\ell+1} + \frac{4sM}{h^2\left(\beta+\frac{1}{s}\right)} \mu^{\ell+1}_{i,j} &=&  S^{(1)}_{i,j} + \frac{sM}{h^2\left(\beta+\frac{1}{s}\right)} \left( \mu_{i,j+1}^{\ell} +\mu_{i,j-1}^{\ell+1} +\mu_{i-1,j}^{\ell+1}+\mu_{i+1,j}^{\ell} \right),
	\nonumber
	\\
&&
	\\
 -\left(\alpha + \left(\phi^{\ell}_{i,j}\right)^2 \right) \phi^{\ell+1}_{i,j} +\mu_{i,j}^{\ell+1} + \frac{4}{h^2} \kappa^{\ell+1}_{i,j} &=& S^{(2)}_{i,j} + \frac{1}{h^2}\left( \kappa_{i,j+1}^{\ell} +\kappa_{i,j-1}^{\ell+1} +\kappa_{i-1,j}^{\ell+1}+\kappa_{i+1,j}^{\ell} \right),
	\\
\frac{4}{h^2}\phi^{\ell+1}_{i,j} + \kappa_{i,j}^{\ell+1} &=& S^{(3)}_{i,j} + \frac{1}{h^2}\left( \phi_{i,j+1}^{\ell} +\phi_{i,j-1}^{\ell+1} +\phi_{i-1,j}^{\ell+1}+\phi_{i+1,j}^{\ell} \right) .
	\end{eqnarray}
At each grid cell $(i,j)$, we solve for the unknowns $\phi^{\ell+1}_{i,j}$, $\mu^{\ell+1}_{i,j}$, $\kappa^{\ell+1}_{i,j}$ using Cramer's rule.  One complete smoothing pass is effected when all of the grid points $(i,j)$ have been visited exactly once.  The smoothing operation is complete when $\ell_{\max}$ full smoothing passes (or smoothing iterations) have been conducted.

	\subsection{Multigrid Method for the Second-Order Convex Splitting Scheme}
	\label{subapp-multigrid-second}

In the case of the fully second order scheme (\ref{mpfc_scheme2_3}) -- (\ref{discrete_2_b}), given $\phi^k,\psi^k,\phi^{k-1} \in \mathcal{C}_{\overline{m} \times \overline{n}}$ we need to find the periodic grid functions $\phi,\psi, \mu, \kappa \in  \mathcal{C}_{\overline{m} \times \overline{n}}$ satisfying
	\begin{eqnarray}
\phi_{i,j} -\phi^k_{i,j} &=&  \frac{s M}{ \beta + \frac{2}{s}} \Delta_h \mu_{i,j} + \frac{2}{  \beta+ \frac{2}{s}  }\psi^k_{i,j},
 	\label{NA2_1}
	\\ 
\mu_{i,j} &=& \frac{1}{2} \chi\left(\phi_{i,j},\phi^k_{i,j} \right) \left( \phi_{i,j}+\phi^k_{i,j} \right) + \frac{\alpha}{2}\left(\phi_{i,j} + \phi^k_{i,j} \right) + 2\Delta_h \tilde{\phi}_{i,j}^{k+\hf} + \Delta_h \kappa_{i,j},
	\label{NA2_2}
	\\ 
\kappa_{i,j} &=& \frac{1}{2} \left( \Delta_h \phi_{i,j} + \Delta_h \phi^{k}_{i,j} \right),
	\label{NA2_3}
	\\ 
\psi_{i,j} &=& \frac{2}{s}\left(\phi_{i,j} -\phi^k_{i,j}\right) -  \psi^k_{i,j},
	\label{NA2_4}
	\end{eqnarray}
where we have dropped the superscripts on the unknowns for the sake of brevity.  As before, we solve for $\phi,\,  \mu,\, \kappa \in \mathcal{C}_{\overline{m} \times \overline{n}}$ using (\ref{NA2_1}) -- (\ref{NA2_3}) first and then update $\psi \in \mathcal{C}_{\overline{m} \times \overline{n}}$ using (\ref{NA2_4}).	
	
With the same notation as before, the nonlinear operator $\mathbf{N}({\bf u}) = (N^{(1)}({\bf u}),N^{(2)}({\bf u}),N^{(3)}({\bf u}))^T$ is defined component-wise via
\begin{eqnarray}
N^{(1)}_{ij} ({\bf u}) &:=& \phi_{ij} - \frac{M s}{\beta+\frac{2}{s}} \Delta_h\mu_{ij},
	\\
N^{(2)}_{ij} ({\bf u}) &:=& \mu_{ij} - \frac{1}{2} \chi\left(\phi_{i,j},\phi^k_{i,j} \right) \left( \phi_{i,j}+\phi^k_{i,j} \right) - \frac{\alpha}{2} \phi_{i,j} - \Delta_h \kappa_{i,j} ,
	\\
N^{(3)}_{ij} ({\bf u}) &:=& \kappa_{ij} - \frac{1}{2} \Delta_h \phi_{ij} ,
	\end{eqnarray}
and the source term $\mathbf{S}$ is defined component-wise via
	\begin{equation}
S^{(1)}_{ij}:=\phi^k_{ij} +\frac{2}{\beta+\frac{2}{s}} \psi^{k}_{ij} , \quad S^{(2)}_{ij} := \frac{\alpha}{2}\phi^{k}_{i,j} + 2\Delta_h \tilde{\phi}^{k+\hf}_{ij} , \quad S^{(3)}_{ij} := \frac{1}{2} \Delta_h \phi^k_{i,j} .
	\end{equation}

The nonlinear Gauss-Seidel smoothing scheme is as follows:
	\begin{eqnarray}
&&\phi_{i,j}^{\ell+1} + \frac{4sM}{h^2\left(\beta+\frac{2}{s}\right)} \mu^{\ell+1}_{i,j} 
	\nonumber
	\\
&& \hspace{0.35in} = S^{(1)}_{i,j} + \frac{4sM}{h^2\left(\beta+\frac{2}{s}\right)} \left( \mu_{i,j+1}^{\ell} +\mu_{i,j-1}^{\ell+1} +\mu_{i-1,j}^{\ell+1}+\mu_{i+1,j}^{\ell} \right) ,
	\\
&&-\left( \frac{\alpha}{2} + \frac{1}{2}\chi\left(\phi^{\ell}_{i,j}, \phi^{k}_{i,j}\right)\right) \phi^{\ell+1}_{i,j} +\mu_{i,j}^{\ell+1} + \frac{4}{h^2} \kappa^{\ell+1}_{i,j} 
	\nonumber
	\\
&& \hspace{0.35in} = S^{(2)}_{i,j} + \frac{1}{2}\chi\left(\phi^{\ell}_{i,j},  \phi^{k}_{i,j}\right) \phi^{k}_{i,j} + \frac{1}{h^2}\left( \kappa_{i,j+1}^{\ell} +\kappa_{i,j-1}^{\ell+1} +\kappa_{i-1,j}^{\ell+1}+\kappa_{i+1,j}^{\ell} \right) ,
	\\
&&\frac{2}{h^2}\phi^{\ell+1}_{i,j} + \kappa_{i,j}^{k+1,\ell+1} 
	\nonumber
	\\
&& \hspace{0.35in} = S^{(3)}_{i,j}+ \frac{1}{2h^2}\left( \phi_{i,j+1}^{\ell} +\phi_{i,j-1}^{\ell+1} +\phi_{i-1,j}^{\ell+1}+\phi_{i+1,j}^{\ell} \right) .
\end{eqnarray}
At each cell $(i,j)$, we solve for the unknowns $\phi^{\ell+1}_{i,j}$, $\mu^{\ell+1}_{i,j}$, $\kappa^{\ell+1}_{i,j}$ using Cramer's rule. One complete smoothing pass is effected when all of the grid points $(i,j)$ have been visited exactly once.  The smoothing operation is complete when $\ell_{\max}$ full smoothing passes (or smoothing iterations) have been conducted.
	
	\clearpage
	\newpage

	\section{Tables}

	\begin{table}[ht]
	\begin{center}
	\begin{tabular}{|c|c|c|c|}
	\hline
$h_c$ & $h_f$ & $\nrm{\phi_{h_f}-\phi_{h_c}}_2$ & Rate
	\\
	\hline
$\frac{32}{32}$ & $\frac{32}{64}$  & $2.4306\times 10^{-4}$ & ---
	\\
	\hline
$\frac{32}{64}$ & $\frac{32}{128}$  & $6.6989 \times 10^{-5}$ & 1.8594
	\\
	\hline
$\frac{32}{128}$ & $\frac{32}{256}$  & $1.7059 \times 10^{-5}$ & 1.9734
	\\
	\hline
$\frac{32}{256}$ & $\frac{32}{512}$  & $4.2828  \times 10^{-6}$ & 1.9939 
	\\
	\hline
	\end{tabular}
\caption{Error (successive differences) and calculated convergence rates for the first-order convex splitting scheme (\ref{scheme1_3}) -- (\ref{discrete_1_b}).  The parameters are taken to be $\epsilon =0.025$ ($\alpha = 1 - \epsilon$) and $\beta =0.9$, and $t_f = 10$. The time step sizes are determined by the quadratic refinement path $s=0.025 h^2$.  The global error is predicted to be $\mathcal{O}(s)+\mathcal{O}(h^2) =\mathcal{O}(h^2)$, which is confirmed in the test.}
	\label{tab1}
	\end{center}
	\end{table}

		\begin{table}[ht]
	\begin{center}
	\begin{tabular}{|c|c|c|c|}
	\hline
$h_c$ & $h_f$ & $\nrm{\phi_{h_f}-\phi_{h_c}}_2$ & Rate
	\\
	\hline
$\frac{32}{32}$ & $\frac{32}{64}$  & $2.5301 \times 10^{-4}$ & --
	\\
	\hline
$\frac{32}{64}$ & $\frac{32}{128}$  & $6.6800 \times 10^{-5}$ & 1.9213
	\\
	\hline
$\frac{32}{128}$ & $\frac{32}{256}$  & $1.6818 \times 10^{-5}$ & 1.9899
	\\
	\hline
$\frac{32}{256}$ & $\frac{32}{512}$  & $4.2099 \times 10^{-6}$ & 1.9981
	\\
	\hline
	\end{tabular}
\caption{Errors (successive differences) and calculated convergence rates for the second-order convex splitting scheme (\ref{mpfc_scheme2_3}) -- (\ref{discrete_2_b}).  The parameters are taken to be $\epsilon =0.025$ ($\alpha = 1 - \epsilon$), $\beta =0.9$, and $t_f = 10$.  The time steps are determined by the linear refinement path $s=0.05 h$.  The global error is predicted to be $\mathcal{O}(s^2)+\mathcal{O}(h^2) =\mathcal{O}(h^2)$, which is confirmed in the test.}
	\label{tab2}
	\end{center}
	\end{table}

	\clearpage
	\newpage

	\section{Figures}

	\begin{figure}[ht]
	\begin{center}
	\includegraphics[width=5.5in]{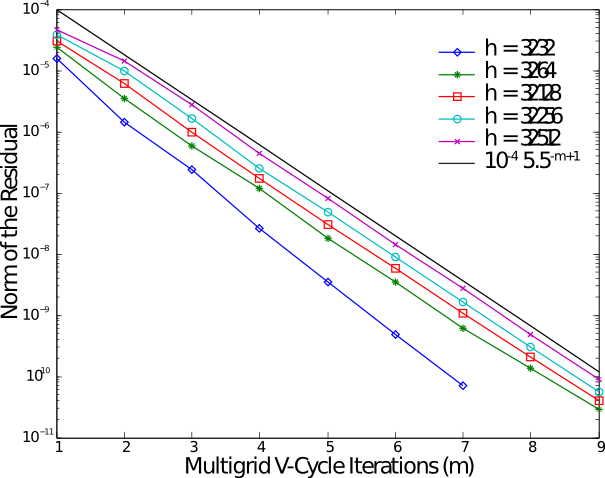}
	\end{center}
\caption{Multigrid convergence test for the second-order convex splitting scheme (\ref{mpfc_scheme2_3}) -- (\ref{discrete_2_b}).  The initial conditions are given in (\ref{initial-data}); the time step size is held fixed at $s = 10$, while $h$ is allowed to vary; and the other parameters are $\epsilon = 0.025$, $M=1$, $\beta = 0.9$, and $\Omega = (0,32)^2$. We use $\ell_{\max}=1$ in the definition of the smoothing operator (cf. App.~\ref{subapp-multigrid-second}), and we conduct exactly 20 time steps, so that the final time $t_f = 200$. The residuals above are reported only for the last time step.  We observe that, independent of the space step size $h$, the norm of the residual is always decreasing by about a factor of $5.5$ for each multigrid v-cycle iteration.  This demonstrates that our multigrid scheme is of nearly optimal order.}
	\label{fig1}	   
	\end{figure}

	\begin{figure}[ht]
	\begin{center}
	\includegraphics[width=6.5in]{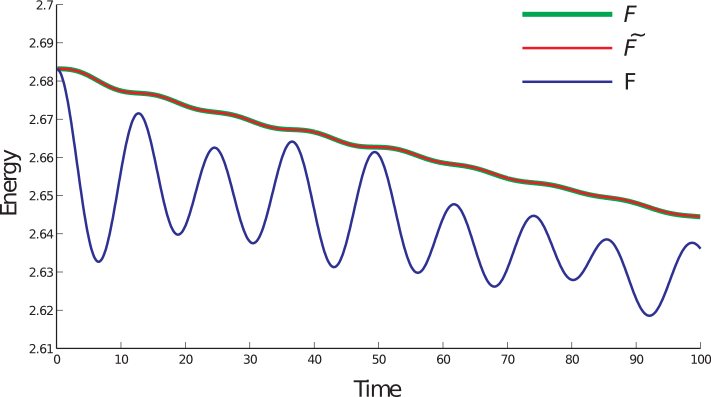}
	\end{center}
\caption{(Color online.) The computed energies $\mathcal{F}$ (\ref{discrete_energy_MPFC}; green), $\widetilde{\mathcal{F}}$ (\ref{energy_numer}; red), and $F$ (\ref{discrete_energy_PFC}; blue) plotted as functions of time.   The computation is effected with our second-order convex-splitting scheme (\ref{mpfc_scheme2_3}) -- (\ref{discrete_2_b}), where the initial data are given in (\ref{initial-data}), and $M=1$, $\epsilon =0.025$ ($\alpha = 1 - \epsilon$), $\Omega = (0,32)\times(0,32)$, $\beta = 0.01$, $h = \frac{32}{128}$, and $s = 0.1$.   As pointed out earlier, the energy $F$ is not necessarily non-decreasing in time.  The only energy that is guaranteed to be non-increasing is $\widetilde{\mathcal{F}}$; this property is observed here.  But, note that $\mathcal{F}$ is nearly identical to $\widetilde{\mathcal{F}}$, which is typically the case.}
	\label{fig2}	   
	\end{figure}

	\begin{figure}
	\begin{center}
	\includegraphics[width=6.5in]{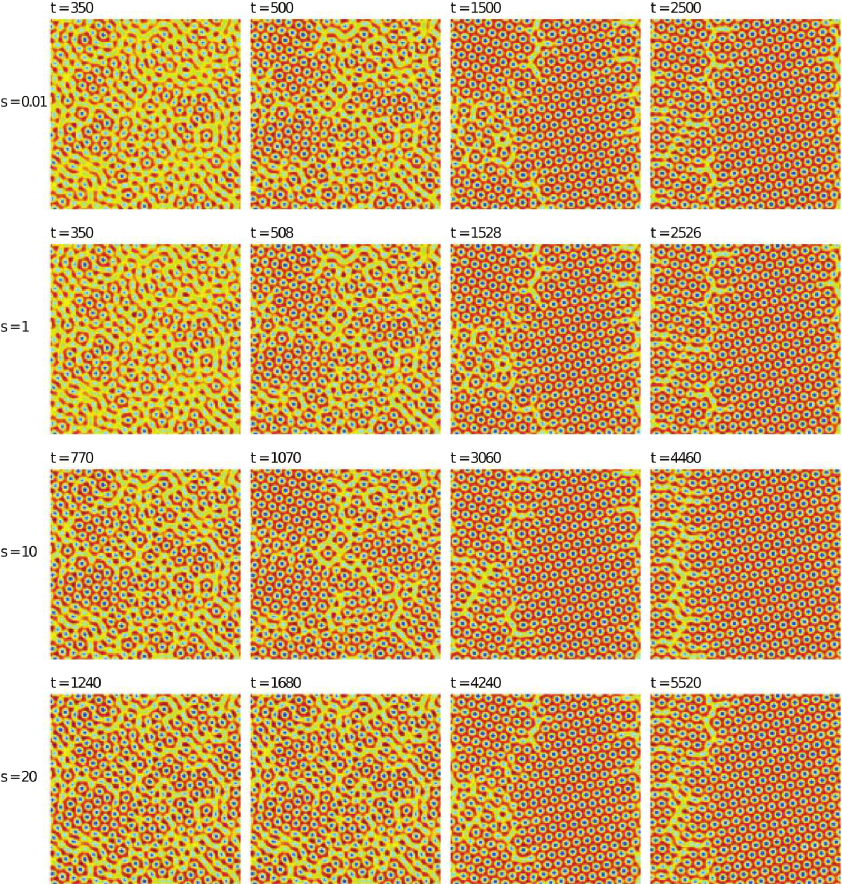}
	\end{center}
\caption{The evolution of the density field $\phi$ calculated using the second-order convex-splitting scheme and four different time steps: $s =0.01$, 1, 10, 20. The parameters are $\epsilon = 0.025$ ($\alpha = 1-\epsilon$), $\beta =0.9$, and $h=1$. The plots in the same column have the same pseudo energy $\mathcal{F}$.}
	\label{fig3}
	\end{figure}

	\begin{figure}
	\begin{center}
\includegraphics[width=4in]{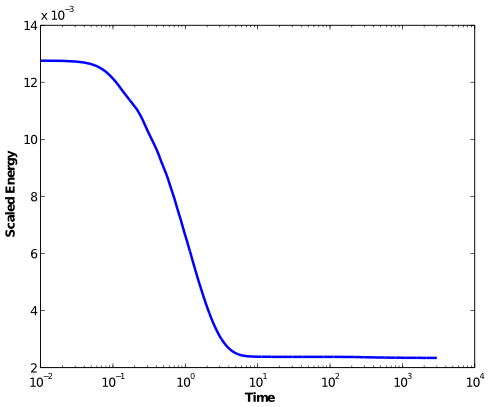}
	\end{center}
\caption{The scaled MPFC pseudo energy $\frac{\mathcal{F}}{L_xL_y}$ as a function of time for the evolution from Fig.~\ref{fig3} using the fully second order scheme with $s=0.01$. Snapshots of the density field $\phi$ are shown in Fig.~\ref{fig3}, row 1.  The numerical and physical parameters for the simulation are given in the caption of Fig.~\ref{fig3} and in the text.}
	\label{fig4}
	\end{figure}

	\begin{figure}
	\begin{center}
\includegraphics[width=6.5in]{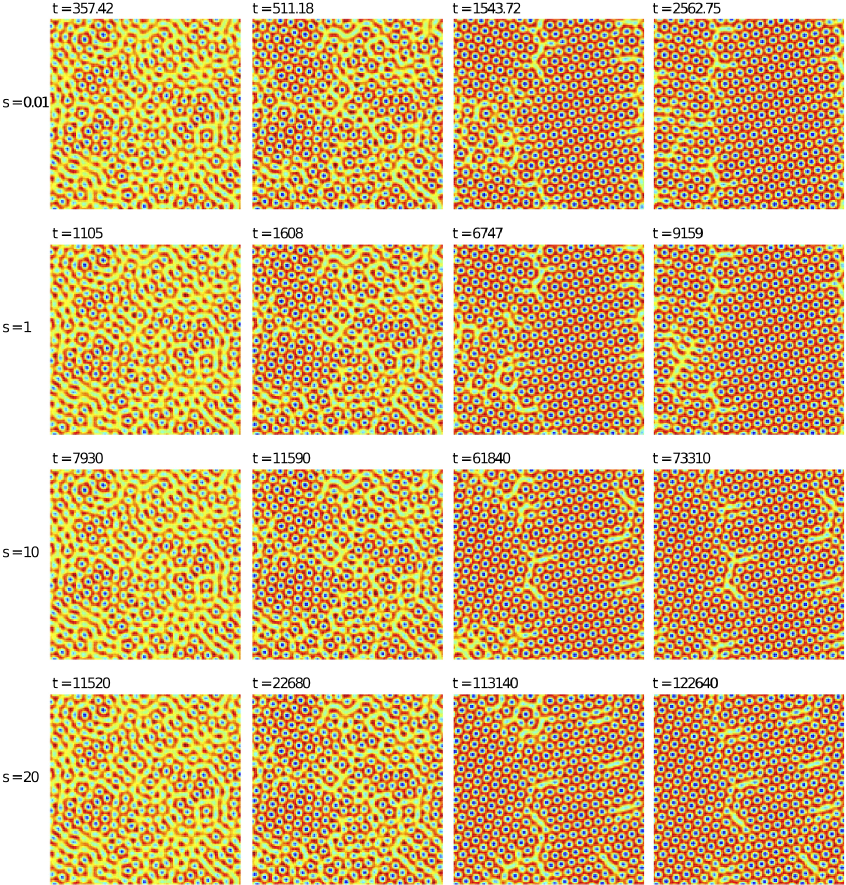}
	\end{center}
\caption{The evolution of the density field $\phi$ calculated using the first-order convex-splitting scheme and four different time steps: $s =0.01$, 1, 10, 20. The parameters are $\epsilon = 0.025$, $\beta =0.9$ and $h=1$. The plots in the same column have the same pseudo energy $\mathcal{F}$. Compare with Fig.~\ref{fig3}.}
	\label{fig5}
	\end{figure}

	\begin{figure}
	\begin{center}
	\includegraphics[width=4in]{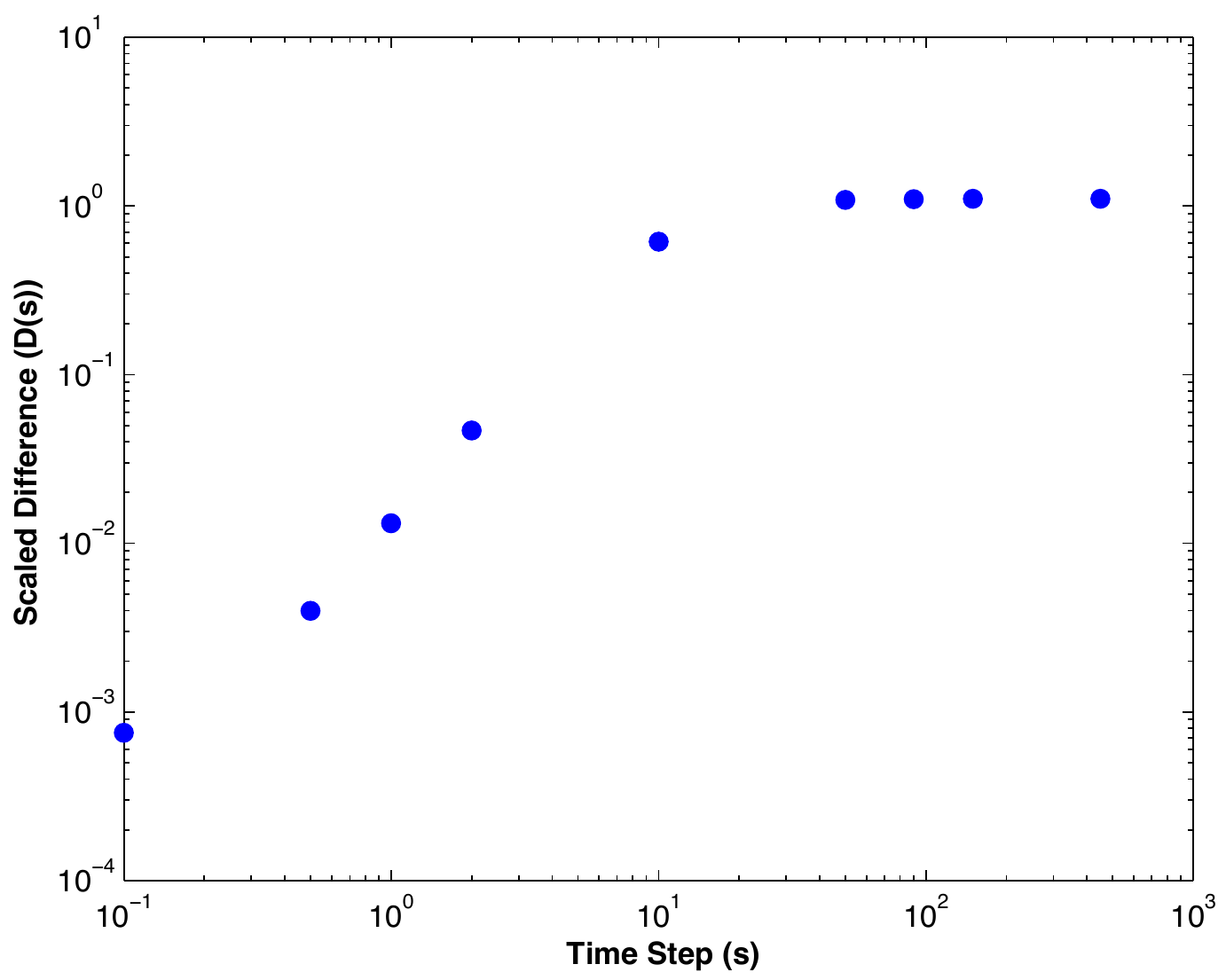}
\caption{Log-log plot of the scaled difference in solutions $D(s)$, defined in (\ref{scaled-difference}), calculated using the second-order convex-splitting scheme.  The parameters are $t_f=350$, $\beta =0.9$, $\epsilon =0.025$ ($\alpha = 1-\epsilon$), and $h = 1$.  The baseline calculation is made using the time step size $s = s_1 = 0.01$; and the corresponding baseline evolution of $\phi$ is shown in Fig.~\ref{fig3}, row 1.}
	\label{fig6}
	\end{center}
	\end{figure}

	\begin{subfigures}

	\begin{figure}
	\begin{center}
	\includegraphics[width=5in]{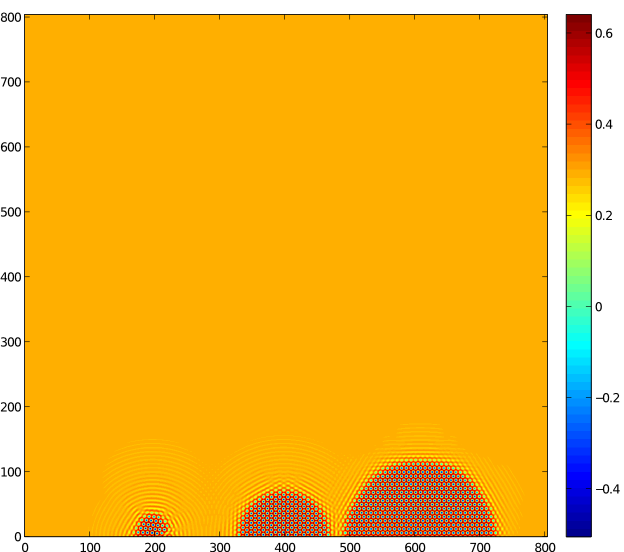}
\caption{Growth of a polycrystal via freezing of super cooled liquid. The figure shows the density field $\phi$ at $t=500$.  The physical  parameters are $M=1$, $\epsilon = 0.25$ ($\alpha = 1-\epsilon$),  $\beta =0.9$, $s=1$ with $h=\frac{804}{2048}$.}
	\label{fig7a}
	\end{center}
	\end{figure}

	\begin{figure}
	\begin{center}
	\includegraphics[width=5in]{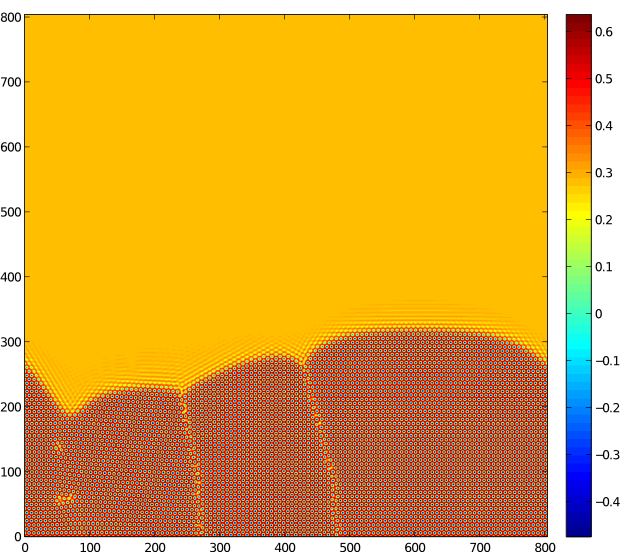}
\caption{Growth of a polycrystal via freezing of super cooled liquid. The figure shows the density field $\phi$ at $t=1000$.  The physical  parameters are $M=1$, $\epsilon = 0.25$ ($\alpha = 1-\epsilon$),  $\beta =0.9$, $s=1$ with $h=\frac{804}{2048}$.}
	\label{fig7b}
	\end{center}
	\end{figure}

	\begin{figure}
	\begin{center}
	\includegraphics[width=5in]{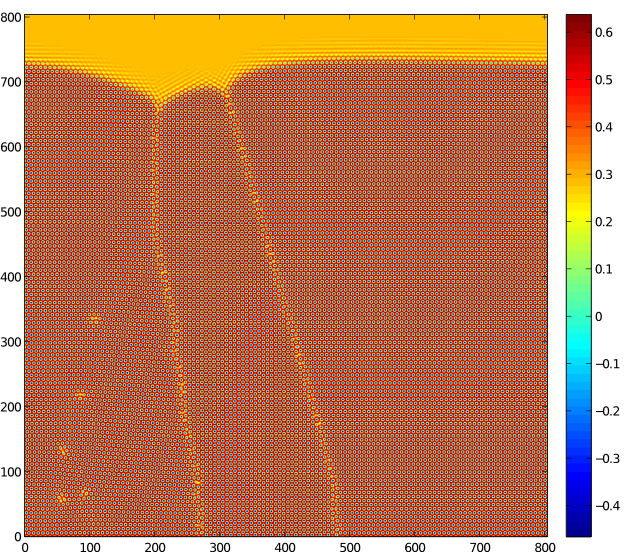}
\caption{Growth of a polycrystal via freezing of super cooled liquid. The figure shows the density field $\phi$ at $t=2000$.  The physical  parameters are $M=1$, $\epsilon = 0.25$ ($\alpha = 1-\epsilon$),  $\beta =0.9$, $s=1$ with $h=\frac{804}{2048}$.}
	\label{fig7c}
	\end{center}
	\end{figure}

	\begin{figure}
	\begin{center}
	\includegraphics[width=5in]{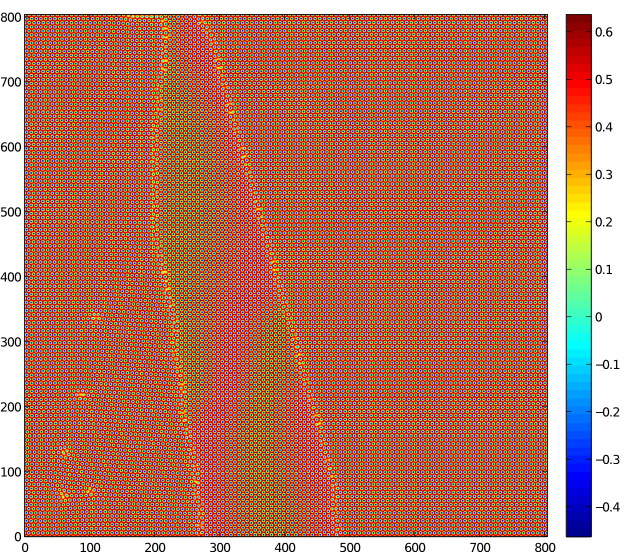}
\caption{Growth of a polycrystal via freezing of super cooled liquid. The figure shows the density field $\phi$ at $t=3000$.  The physical  parameters are $M=1$, $\epsilon = 0.25$ ($\alpha = 1-\epsilon$),  $\beta =0.9$, $s=1$ with $h=\frac{804}{2048}$.}
	\label{fig7d}
	\end{center}
	\end{figure}

	\end{subfigures}

	\begin{figure}
	\begin{center}
	\includegraphics[width=5in]{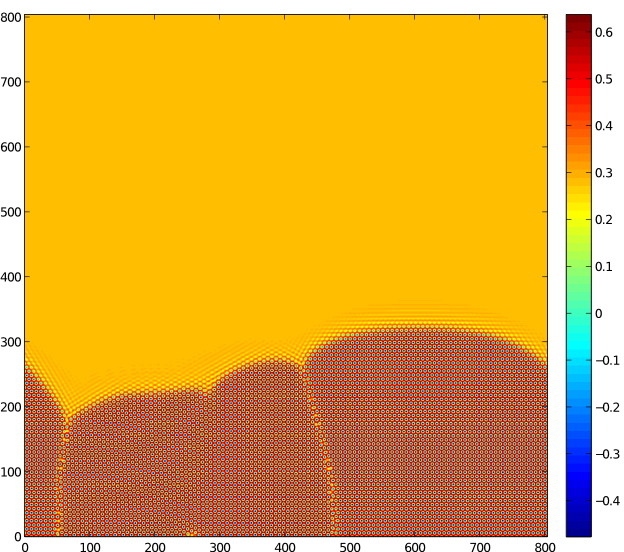}
\caption{Growth of a polycrystal via freezing of super cooled liquid. The figure shows the density field $\phi$ at $t=3000$.  The physical  parameters are $M=1$, $\epsilon = 0.25$ ($\alpha = 1-\epsilon$),  $\beta =10$, $s=1$ with $h=h=\frac{804}{2048}$.  Note the larger value of $\beta$ used here compared with the value used in the simulation represented in Figs.~\ref{fig7a} -- \ref{fig7d}.  Note the near-absence of a grain boundary between the first and second grains (about $x = 300$), compared with the microstructure shown in Fig.~\ref{fig7b}.}
	\label{fig8}
	\end{center}
	\end{figure}

	\begin{figure}
	\begin{center}
	\includegraphics[width=4in]{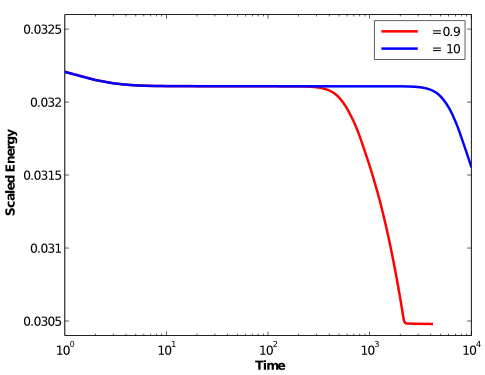}
\caption{(Color online.) Scaled pseudo energy $\frac{\mathcal{F}}{L_x L_y}$ as a function of time for the polycrystal simulations depicted in Figs.~\ref{fig7a} -- \ref{fig7d} and in Fig.~\ref{fig8}.  The red line corresponds to the parameter $\beta = 0.9$ (Figs.~\ref{fig7a} -- \ref{fig7d}), and the blue line corresponds to $\beta =10$ (Fig.~\ref{fig8}). Here the time axis is on the logarithmic scale.}
	\label{fig9}
	\end{center}
	\end{figure}

	\begin{figure}
	\centering
\includegraphics[width=4in]{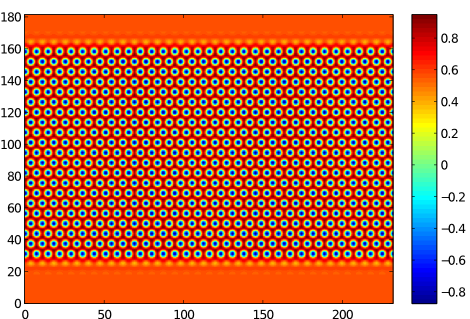}
\caption{An annealed periodic crystal strip, about 20 atom-layers thick, placed in a bath of coexistence liquid.  For the next test,  the results of which are shown in Fig.~\ref{fig11}, we simultaneously apply a horizontal shear strain to the bottom row of atoms and pin the top row of atoms.}
	\label{fig10}
	\end{figure}

	\begin{figure}
	\centering
\includegraphics[width=5.0in]{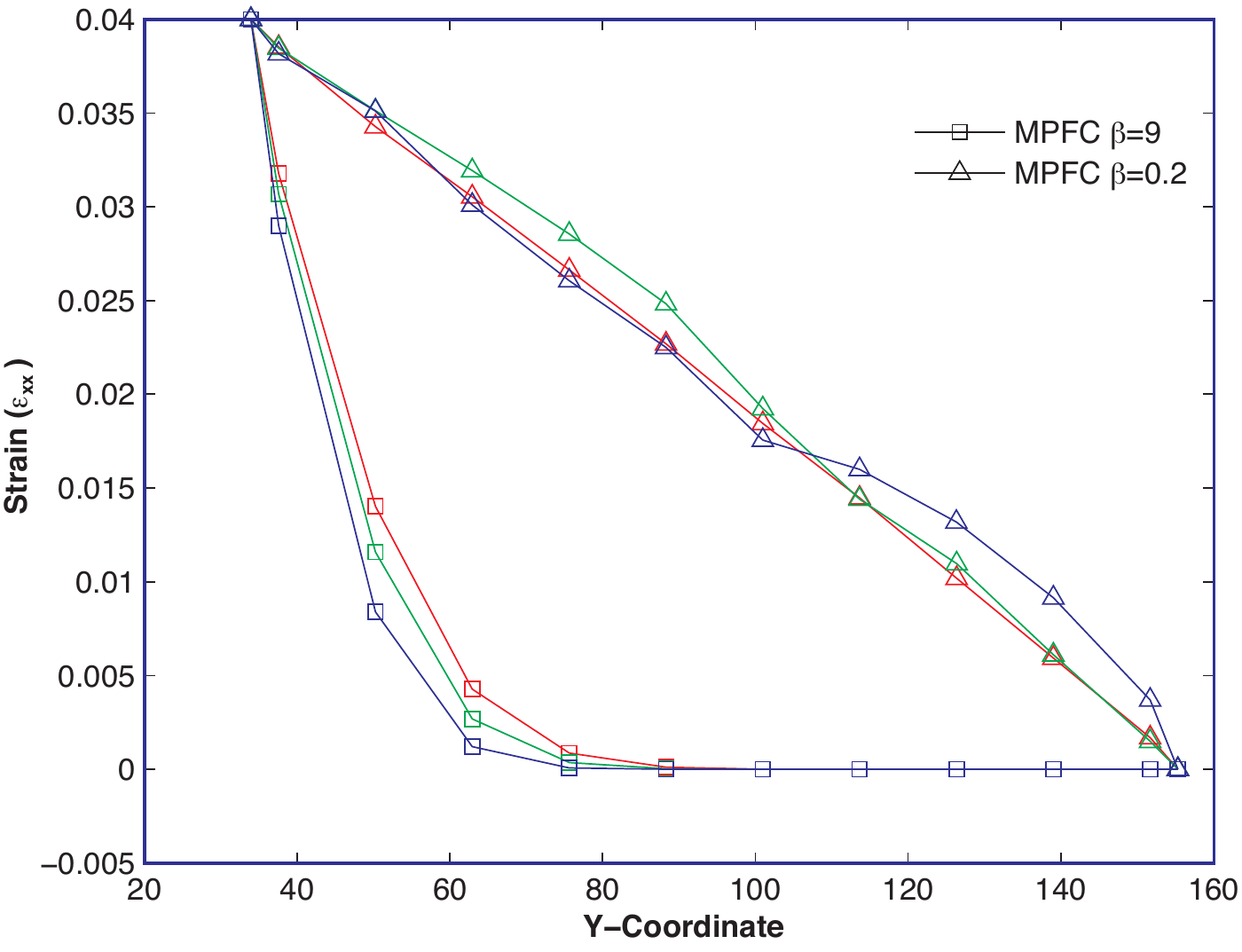}
\caption{(Color online.) Strain ($\varepsilon_{xx}$) fields as functions of the crystal thickness ($y$-coordinate) at times $t=10$ (blue), $t=14$ (green) and, $t=18$ (red).  A horizontal shear strain of 4\% has been applied to the bottom most row of the crystal at $t=0$.  Simulation results are shown for two cases: $\beta =0.2$ (triangles, low damping case) and $\beta = 9$ (squares, high damping case).  Note the oscillatory nature of the low damping MPFC case (triangles). In this model the crystal rapidly relaxes to the strained equilibrium state.  The high damping MPFC case behaves much more like the PFC model, as expected, with a slow relaxation to equilibrium and no oscillatory overshoot at long times.}
	\label{fig11}
	\end{figure}

	\clearpage
	\newpage

	\bibliographystyle{plain}
	\bibliography{mpfcjcp}

	\end{document}